\documentclass[a4paper, 10pt, reqno]{amsart}

\usepackage[utf8]{inputenc}
\usepackage[T1]{fontenc}
\usepackage{lmodern}
\usepackage{microtype}
\usepackage{amsmath}
\usepackage{amsthm}
\usepackage{amssymb}
\usepackage{amsfonts}
\usepackage{thmtools}
\usepackage{relsize}
\usepackage[mathscr]{eucal}
\usepackage[linktocpage]{hyperref}
\usepackage[sort,capitalize]{cleveref}
\usepackage{enumitem}			%Elegante Nummerierung
\usepackage[dvipsnames]{xcolor}
\usepackage[msc-links, abbrev, non-sorted-cites]{amsrefs}
\usepackage{todonotes}
\usepackage{caption}

\makeatletter
\def\nonumberfootnote{\xdef\@thefnmark{}\@footnotetext}			%Fußnote ohne Nummer
\makeatother

\definecolor{colorred}{HTML}{B00000}
\definecolor{colorgreen}{HTML}{258300}
\definecolor{colorblue}{HTML}{2e32fa}
\definecolor{coloryellow}{HTML}{cbbb1a}
\hypersetup{colorlinks=true, linkcolor=colorred, citecolor=colorgreen, urlcolor=coloryellow, pdfusetitle=true}

\numberwithin{equation}{section}

\newcommand{\nlb}{{\ensuremath{\textnormal{b}}}}

\newcommand{\nlC}{{\ensuremath{\textnormal{C}}}}

\newcommand{\rmd}{{\ensuremath{\mathrm{d}}}}

\newcommand{\sfd}{{\ensuremath{\mathsf{d}}}}

\newcommand{\sfn}{{\ensuremath{\mathsf{n}}}}

\newcommand{\sfD}{{\ensuremath{\mathsf{D}}}}

\newcommand{\sfL}{{\ensuremath{\mathsf{L}}}}

\newcommand{\sfS}{{\ensuremath{\mathsf{S}}}}
\newcommand{\sfT}{{\ensuremath{\mathsf{T}}}}

\newcommand{\sfX}{{\ensuremath{\mathsf{X}}}}
\newcommand{\sfY}{{\ensuremath{\mathsf{Y}}}}

\newcommand{\scrB}{{\ensuremath{\mathscr{B}}}}
\newcommand{\scrC}{{\ensuremath{\mathscr{C}}}}

\newcommand{\scrF}{{\ensuremath{\mathscr{F}}}}

\newcommand{\scrL}{{\ensuremath{\mathscr{L}}}}
\newcommand{\scrM}{{\ensuremath{\mathscr{M}}}}
\newcommand{\scrN}{{\ensuremath{\mathscr{N}}}}

\newcommand{\scrT}{{\ensuremath{\mathscr{T}}}}

\newcommand{\bdDelta}{{\ensuremath{\boldsymbol{\Delta}}}}

\newcommand{\N}{\boldsymbol{\mathrm{N}}}						%Natürliche Zahlen
\newcommand{\R}{\boldsymbol{\mathrm{R}}}						%Reelle Zahlen
\renewcommand{\S}{\boldsymbol{\mathrm{S}}}						%Sphäre
\renewcommand{\d}{\,\mathrm{d}}				%Kleines Integral-d

\let\limsup\undefined
\let\liminf\undefined

\DeclareMathOperator*{\limsup}{limsup}		%Limes superior
\DeclareMathOperator*{\liminf}{liminf}		%Limes inferior
\DeclareMathOperator*{\esssup}{esssup}		%Essentielles Supremum
\DeclareMathOperator{\supp}{spt}			%Träger
\theoremstyle{definition}
\newtheorem{bump}{Bump}[section]

\theoremstyle{plain}
\newtheorem{theorem}[bump]{Theorem}
\newtheorem{proposition}[bump]{Proposition}
\newtheorem{definition}[bump]{Definition}
\newtheorem{lemma}[bump]{Lemma}
\newtheorem{corollary}[bump]{Corollary}

\newtheorem{conjecture}[bump]{Conjecture}

\theoremstyle{remark}
\newtheorem{remark}[bump]{Remark}
\newtheorem{example}[bump]{Example}

\crefname{theorem}{Theorem}{Theorems}
\crefname{proposition}{Proposition}{Propositions}
\crefname{definition}{Definition}{Definitions}
\crefname{lemma}{Lemma}{Lemmas}
\crefname{corollary}{Corollary}{Corollaries}
\crefname{hypothesis}{Hypothesis}{Hypotheses}
\crefname{remark}{Remark}{Remarks}
\crefname{example}{Example}{Examples}
\crefname{notation}{Notation}{Notations}
\crefname{conjecture}{Conjecture}{Conjectures}

\renewenvironment{example}
  {\begin{oldexample}}
  {\hfill $\blacksquare$\end{oldexample}}

\renewenvironment{remark}
  {\begin{oldremark}}
  {\hfill $\blacksquare$\end{oldremark}}

\crefformat{section}{{§}#2#1#3}
\crefformat{subsection}{{§}#2#1#3}
\crefformat{subsubsection}{{§}#2#1#3}
\crefformat{appendix}{{§}#2#1#3}

\crefmultiformat{theorem}{Theorems #2#1#3}{ and #2#1#3}{, #2#1#3}{, and #2#1#3}
\crefmultiformat{proposition}{Propositions #2#1#3}{ and #2#1#3}{, #2#1#3}{, and #2#1#3}
\crefmultiformat{definition}{Definitions #2#1#3}{ and #2#1#3}{, #2#1#3}{, and #2#1#3}
\crefmultiformat{lemma}{Lemmas #2#1#3}{ and #2#1#3}{, #2#1#3}{, and #2#1#3}
\crefmultiformat{corollary}{Corollaries #2#1#3}{ and #2#1#3}{, #2#1#3}{, and #2#1#3}
\crefmultiformat{hypothesis}{Hypotheses #2#1#3}{ and #2#1#3}{, #2#1#3}{, and #2#1#3}
\crefmultiformat{remark}{Remarks #2#1#3}{ and #2#1#3}{, #2#1#3}{, and #2#1#3}
\crefmultiformat{example}{Examples #2#1#3}{ and #2#1#3}{, #2#1#3}{, and #2#1#3}
\crefmultiformat{notation}{Notations #2#1#3}{ and #2#1#3}{, #2#1#3}{, and #2#1#3}

\crefmultiformat{section}{{§§}#2#1#3}{ and #2#1#3}{, #2#1#3}{, and #2#1#3}
\crefmultiformat{subsection}{{§§}#2#1#3}{ and #2#1#3}{, #2#1#3}{, and #2#1#3}
\crefmultiformat{subsubsection}{{§§}#2#1#3}{ and #2#1#3}{, #2#1#3}{, and #2#1#3}

\crefrangeformat{equation}{#3\textcolor{black}{(}#1\textcolor{black}{)}#4 to #5\textcolor{black}{(}#2\textcolor{black}{)}#6}

\newcommand{\mms}{\mathsf{M}}				%Basisraum
\newcommand{\met}{\sfd}						%Metrik

\newcommand{\meas}{\mathfrak{m}}				%Referenzmaß

\newcommand{\Leb}{\mathscr{L}}				%Lebesgue-Maß	
\newcommand{\Prob}{\mathscr{P}}		        %Wahrscheinlichkeitsmaße
\newcommand{\Id}{\mathrm{Id}}				%Identitätsabbildung

\newcommand{\bounded}{\nlb}					%Beschränkt
\newcommand{\pr}{\mathrm{pr}}				%Projektion
\newcommand{\Cont}{\nlC}					%Stetige Funktionen

\newcommand{\Lip}{\mathrm{Lip}}				%Lipschitz-stetige Funktionen/Kurven

\DeclareMathOperator{\diam}{diam}			%Diameter
\newcommand{\push}{\sharp}					%Bildmaßoperator

\newcommand{\tsep}{\tau}

\usepackage{blindtext}

\allowdisplaybreaks

\let\oldtocsection=\tocsection

\let\oldtocsubsection=\tocsubsection

\let\oldtocsubsubsection=\tocsubsubsection

\renewcommand{\tocsection}[2]{\hspace{0em}\oldtocsection{#1}{#2}}
\renewcommand{\tocsubsection}[2]{\hspace{1em}\oldtocsubsection{#1}{#2}}
\renewcommand{\tocsubsubsection}[2]{\hspace{2em}\oldtocsubsubsection{#1}{#2}}

\newcommand{\nocontentsline}[3]{}
\newcommand{\tocless}[2]{\bgroup\let\addcontentsline=\nocontentsline#1{#2}\egroup}

\allowdisplaybreaks

\makeatletter
\@namedef{subjclassname@2020}{\textup{2020} Mathematics Subject Classification}
\makeatother

\newcommand{\G}{\boldsymbol{\mathrm{G}}}

\newcommand{\MM}{\boldsymbol{\mathrm{M}}}

\DeclareMathOperator{\Par}{Par}

\newcommand{\BOX}{\square}
\newcommand{\dis}{\mathrm{dis}}
\newcommand{\LL}{\sfL}

\newcommand{\PPP}{{\boldsymbol{\mathrm{P}}}}

\begin{document}

\title[Gromov's reconstruction theorem]{Gromov's reconstruction theorem and\\ measured Gromov--Hausdorff convergence\\ in Lorentzian geometry}%required
\author{Mathias Braun}%required
%\contrib[...]{...}%optional
\address{Institute of Mathematics, EPFL, 1015 Lausanne, Switzerland}%required
%\curraddr{...}%optional
\email{\href{mailto:mathias.braun@epfl.ch}{mathias.braun@epfl.ch}}%optional
%\urladdr{...}%optional
%\dedicatory{...}%optional

\author{Clemens Sämann}%required
%\contrib[...]{...}%optional
\address{Faculty of Mathematics, University of Vienna, Oskar-Morgenstern-Platz 1, 1090 Vienna, Austria}%required
%\curraddr{...}%optional
\email{\href{clemens.saemann@univie.ac.at}{clemens.saemann@univie.ac.at}}%optional
%\urladdr{...}%optional
%\dedicatory{...}%optional

\date{\today}

\maketitle

\begin{abstract} We establish Gromov's celebrated  reconstruction theorem in Lo\-rentz\-ian geometry. Alongside this result, we introduce and study a natural concept of isomorphy of normalized bounded Lorentzian metric measure spaces. We  outline applications to the spacetime reconstruction problem from causal set theory. Lastly, we propose three notions of convergence of (isomorphism classes of) normalized bounded Lorentzian metric measure spaces, for which we prove several fundamental properties. 
\end{abstract}

\tableofcontents

\thispagestyle{empty}

\addtocontents{toc}{\protect\setcounter{tocdepth}{2}}

\newpage
\section{Introduction}\label{Ch:Intro}

A new approach to spacetime geometry from the angle of metric geometry was  pioneered by Kunzinger--Sämann \cite{KS:18} after earlier works of Kron\-heimer--Penrose \cite{KP:67} and Busemann \cite{Bus:67}. For recent reviews that summarize a large number of key contributions to this endeavor, we refer to Cavalletti--Mondino \cite{CM:22}, Sämann \cite{Sae:24}, McCann \cite{McC:25}, and Braun \cite{Bra:25}. 

This article adds new aspects to this theory. Our major result is a version of Gromov's metric reconstruction theorem  \cite{Gro:99} in Lorentzian signature. Along with adjacent literature, this is summarized in \cref{Sub:GromRec}. Our adaptation triggers interesting applications to causal set theory outlined in \cref{Sub:SptReconst}, notably its  Hauptvermutung (viz.~fundamental conjecture) as well as an approach to a conjecture of Bombelli \cite{Bom:00} which will be addressed in future work. The Lorentzian Gromov reconstruction theorem comes with (and demonstrates the strength of) a new notion of isomorphy of bounded Lorentz\-ian metric spaces endowed with a probability reference measure. It is inspired by isomorphy of metric measure spaces, a fruitful concept Gromov promoted in his green book \cite{Gro:99}. Here, bounded Lorentzian metric spaces are a ``metric'' quantification of spacetimes set up by Minguzzi--Suhr \cite{MS:22b},  alternative to Kunzinger--Sämann's Lorentzian pre-length spaces \cite{KS:18}.

Our second topic concerns measured Lorentz--Gromov--Hausdorff convergence, a desideratum already expressed 25 years ago by  Bombelli \cite{Bom:00}.  Starting from our Gromov reconstruction theorem, we introduce and study three notions of measured Gromov--Haus\-dorff convergence for sequences of normalized bounded Lorentzian metric spaces, cf.~\cref{Sub:MLGHCVG}.  Prior  approaches to that were given by Cavalletti--Mondino \cite{CM:20} and Mondino--Sämann \cite{MS:25} (of different kind in different settings). With our concept of isomorphy and its properties, we present a first systematic access to the measured Lorentz--Gromov--Hausdorff topology. This complements the preceding systematic approaches to Lorentz--Gromov--Hausdorff convergence without reference  measures by Müller \cite{Mue:22}, Minguzzi--Suhr \cite{MS:22b}, and others summarized in \cref{Sub:MLGHCVG} below.  A  thorough study of the ``space of metric measure space\-times'', notably its topology and geometry --- which reflects  the considerations of Sturm \cite{Stu:12} in Riemannian signature preceded by several works summarized in \cref{Sub:MLGHCVG} --- is naturally motivated from our results. This is outsourced to future work.

\subsection{Gromov's reconstruction theorem in Lorentzian geometry}\label{Sub:GromRec} Gromov's reconstruction theorem \cite{Gro:99} is a central result in the geometry of metric measure spaces. Roughly speaking, it asserts the following:  two normalized metric measure spaces are isomorphic if and only if for any two families of i.i.d.~random variables of the same fixed yet arbitrary cardinality $k\in\N$ --- sampled from the spaces in question --- the induced \emph{finite} metric measure spaces coincide in law. A beautiful probabilistic proof from \cite{Gro:99}  credited to Vershik employs the strong law of large numbers;  Kondo \cite{Kon:05} later filled in some technical details.  Vershik  studied random metric spaces and random matrices in detail, e.g.~\cite{Ver:02,Ver:03,Ver:04}. Some branches of applied mathematics have adopted these insights, e.g.~graph theory \cite{Lov:12},  image processing \cite{BK:04},  or topological data science \cite{BM:13}. 

Our first objective is to adapt Gromov's reconstruction theorem to Lorentz\-ian geometry. Before coming to its formulation, we comment on our setting and the required notion of isomorphy.

Throughout the sequel, we will work on \emph{normalized\footnote{With few simple changes, all results in our paper extend to punctures or bounded Lorentzian metric spaces endowed with arbitrary Borel reference measures,  which are necessarily finite qua our setting, cf. \cref{Re:Finite meas sp!}.} bounded Lorentzian metric measure spaces}. These are triples $\scrM := (\mms,\tau,\meas)$, where $\meas$ is a Borel probability measure added to a bounded Lorentzian metric space $(\mms,\tau)$ in the sense of  Min\-guzzi--Suhr \cite{MS:22b}. Here $\smash{\tau\colon\mms^2\to \R_+}$ is a function which, besides other compatibility conditions, satisfies the reverse triangle inequality akin to the time separation function of a spacetime, cf.~\cref{Def:Bounded Lorentzian metric spaces}. We do not assume $\meas$ has full support. The structures from \cite{MS:22b} have two advantages we frequently use: their topology is nonambiguous and isometries (maps which preserve the time separation functions in question, cf.~\cref{Def:Isometryy}) have  good continuity properties. Yet, we believe  with suitable modifications our considerations should extend to related  synthetic approaches to nonsmooth spacetimes such as Kunzinger--Sämann \cite{KS:18}, McCann \cite{McC:24}, Braun--McCann \cite{BM:23}, Beran et al.~\cite{BBC+:24}, or Bykov--Minguzzi--Suhr \cite{BMS:24}.

In \cref{Def:Isomorphy2}, we then introduce the notion of \emph{isomorphy} of two bounded Lorentzian metric measure spaces $\scrM$ and $\scrM'$. Roughly speaking, this is a distance-preserving map $\iota\colon\supp\meas\to \supp\meas'$ such that $\iota_\push\meas =\meas'$, where $\supp$ denotes the support of a Borel measure and $._\push$ the push-forward. This mimics Gromov's isomorphy of metric measure spaces \cite{Gro:99}.  Some technical care, however, is required in setting up this notion. Recall the inclusions $\supp\meas\subset\mms$ or  $\supp\meas'\subset\mms'$ may be proper\footnote{In many relevant cases of interest, such as spacetimes with their canonical volume measures, full support is however granted.}. Unlike subsets of metric spaces --- which trivially stay metric spaces --- proper subsets of bounded Lorentzian metric spaces are not necessarily bounded Lorentzian metric spaces. This will be compensated by taking distance quotients, cf.~\cref{Re:Distquot}.  After this  process, iso\-mor\-phisms enjoy good continuity properties shown in \cref{Pr:Homeomorphy} combined with results of Minguzzi--Suhr \cite{MS:22b}. Isomorphy will be an equivalence relation on the class of normalized bounded Lorentzian metric measure spaces. We also establish a simple  equivalence in terms of appropriate \emph{couplings} in \cref{Le:Char iso}:  $\scrM$ and $\scrM'$ are isomorphic if and only if their reference measures $\meas$ and $\meas'$ admit a coupling $\pi$ such that $\tau(x,y) = \tau'(x',y')$ for $\smash{\pi^{\otimes 2}}$-a.e.~$\smash{(x,x',y,y')\in (\mms\times\mms')^2}$. This does  not involve any quotient construction. Lastly, in \cref{Le:Extension}, we show an extension theorem for isometries defined on dense subsets (which is not obvious in our Lorentzian signature).

Let $\scrM$ be a normalized bounded Lorentzian metric measure space. For  $k\in \N$, we define $\smash{\sfT^k\colon \mms^k\to \R^{k\times k}}$ componentwise by
\begin{align}\label{Eq:T^kdef}
\sfT^k(x_1,\dots,x_k)_{ij} := \tsep(x_i,x_j).
\end{align}
By \cref{Def:Bounded Lorentzian metric spaces}, this map takes values in the closed set of $k\times k$-matrices
\begin{align}\label{Eq:Gkdef}
\begin{split}
    \G^k &:= \big\{a\in \R^{k\times k} : \textnormal{for every }i,j,l\in\{1,\dots,k\} \textnormal{ we have } a_{ii} = 0,\\
     &\qquad\qquad a_{ij} \geq 0, \text{ and}\\
     &\qquad\qquad \textnormal{if } a_{ij}>0 \textnormal{ and } a_{jl}>0\textnormal{ then }a_{ij} + a_{jl} \leq a_{il}\big\}.
     \end{split}
\end{align}

\begin{definition}[Polynomial]\label{Def:Poly} A \emph{polynomial} is a function $\Phi$ defined on the class of normalized bounded Lorentzian metric measure spaces of the form
\begin{align*}
\Phi(\scrM) := \int_{\mms^k}\varphi\circ\sfT^k\d\meas^{\otimes k}
\end{align*}
for some given $k\in\N$ and some given $\smash{\varphi\in\Cont_\bounded(\R^{k\times k})}$.
\end{definition}

\begin{theorem}[Lorentzian Gromov reconstruction theorem]\label{Th:Gromov reconstruction} Assume $\scrM$ and $\scrM'$ are normalized bounded Lorentzian metric measure spaces. Then $\scrM$ and $\scrM'$ are isometric if and only if $\Phi(\scrM) = \Phi(\scrM')$ for every polynomial $\Phi$.
\end{theorem}

The proof of \cref{Th:Gromov reconstruction} will be given in \cref{Sub:GromovRecon}. It is inspired by Vershik's proof from Gromov's book \cite{Gro:99}, modulo several peculiarities specific to our Lorentzian setting we collect in  \cref{Sub:BLMSsumm}.

As its metric counterpart, our  \cref{Th:Gromov reconstruction} has an elegant stochastic  form. Let $(X_i)_{i\in\N}$ be i.i.d.~$\mms$-valued random variables with $X_1\sim \meas$. Then for every $k\in\N$, the random matrix $\smash{\sfT^k(X_1,\dots,X_k)}$ has law
\begin{align}\label{Eq:barm^k}
\bar{\meas}^k := \sfT^k_\push\meas^{\otimes k}.
\end{align}
Here $\sharp$ is the usual push-forward operation, cf.~\cref{Sub:Cvg prob meas}. Modulo a distance quotient, $\smash{\sfT^k(X_1,\dots,X_k)}$ represents a \emph{random causal set} (cf.\ Subsection \ref{Sub:SptReconst}); thus, it constitutes a Lorentzian version of Vershik's random metric spaces \cite{Ver:02,Ver:03,Ver:04}. Sampling i.i.d.~$\mms'$-valued random variables $\smash{(X_i')_{i\in\N}}$ with $\smash{X_1'\sim\meas'}$ and modeling everything on a common probability space $(\Omega,\scrF,\PPP)$, \cref{Th:Gromov reconstruction} translates into the following.

\begin{theorem}[Lorentzian Gromov reconstruction theorem, probabilistic version]\label{Th:LG prob} Assume $\scrM$ and $\scrM'$ are normalized bounded Lorentzian metric measure spaces. Then $\scrM$ and $\scrM'$ are isomorphic if and only if for every $k\in\N$, the random matrices $\smash{\sfT^k(X_1,\dots,X_k)}$ and $\smash{\sfT'^k(X_1',\dots,X_k')}$ coincide in law, i.e.
\begin{align*}
\sfT^k(X_1,\dots,X_k)_\push\PPP = \sfT'^k(X_1',\dots,X_k')_\push\PPP.
\end{align*}
\end{theorem}

\subsection{Measured Gromov--Hausdorff convergence in Lorentzian geometry}\label{Sub:MLGHCVG} Setting up a   Lorentzian analog of the successful Gromov--Hausdorff theory from metric (measure) geometry --- cf.~Gromov \cite{Gro:99} and Burago--Burago--Ivanov \cite{BBI:01} for background --- is a central research area in nonsmooth spacetime geometry. We refer e.g.~to the open problems section of Cavalletti--Mondino's review \cite{CM:22} or the introductions of Minguzzi--Suhr \cite{MS:22b} and  Sakovich--Sormani \cite{SS:24}.

Several different notions of such convergence have been defined for (abstract generalizations of)  spacetimes which do not involve reference measures. Inspired by earlier works of Noldus \cite{Nol:04} and Bombelli--Noldus \cite{BN:04}, Müller \cite{Mue:22} and Minguzzi--Suhr \cite{MS:22b} independently introduced Lorentz--Gromov--Hausdorff convergence of Cauchy slabs and bounded Lorentzian metric spaces, respectively, by mimicking the distortion approach to the Gromov--Hausdorff distance. The theory of \cite{MS:22b} was later extended to the unbounded case by Bykov--Minguzzi--Suhr \cite{BMS:24}. Alternative unmeasured Gromov--Hausdorff theories in terms of the null distance of Sormani--Vega \cite{SV:16} were set up by  Allen--Burtscher \cite{AB:22}, Kunzinger--Steinbauer \cite{KS:22}, and Sakovich--Sormani \cite{SS:24}. Although in line with metric geometry, these works use more than merely the time separation functions in question. Yet another  recent proposal using $\varepsilon$-nets of causal diamonds comes from Mondino--Sämann \cite{MS:25}.

The case including reference measures is sparser. Adding them to the given data, however, is relevant for various reasons.  First, as pointed out by McCann \cite{McC:20}, Mondino--Suhr \cite{MS:22}, and Cavalletti--Mondino \cite{CM:20}, they are natural to synthesize timelike Ricci curvature bounds with  \cite{CM:20} (see also Braun \cite{Bra:22b} and Braun--McCann \cite{BM:23}). Second, they naturally give rise to time functions, cf.~e.g.~Geroch \cite{Ger:70}, one may want to control along sequences of spacetimes. Third, as explained in  \cref{Sub:RelGromHaus}, their couplings yield a simple way to construct relations between the spaces in question. Fourth, reference measures connect the abstract theory we deal with to sprinklings of spacetimes by samples of Poisson point processes, as detailed in \cref{Sub:SptReconst}. Using such Poisson sprinklings, Bombelli \cite{Bom:00} introduced a Hellinger-type distance between two spacetimes of finite volume. In a withdrawn preprint \cite{BNT:12}, Bombelli--Noldus--Tafoya constructed Gromov--Hausdorff distances between causal measure spaces by constructing metrics from the volume of causal diamonds. Cavalletti--Mondino \cite{CM:20} defined measured Lorentz--Gromov--Hausdorff convergence inspired by Gromov's union lemma \cite{Gro:99}. It appears their  embeddings  are difficult  to construct in general, even for Chru\'sciel--Grant approximations of continuous spacetimes \cite{CG:12}. Recently, Mondino--Sämann \cite{MS:25} contributed another approach alongside a precompactness theorem based on uniform bounds on the cardinality of $\varepsilon$-nets of causal diamonds. 

In Riemannian signature, measured Gromov--Hausdorff convergence has received considerable attention. The first work developing such a notion was Fukaya  \cite{Fuk:87} in studying spectral convergence for Laplacians. Starting from Gromov's book \cite{Gro:99}, the measured Gromov--Hausdorff topology on the set  of isomorphism classes of metric measure spaces has been studied by Greven--Pfaffelhuber--Winter \cite{GPW:09}, Löhr \cite{Loe:13}, and Gigli--Mondino--Savaré \cite{GMS:15}. We also refer the  reader to the monograph of Shioya \cite{Shioya2016} (with an emphasis on the related concentration of measure phenomenon of  Lévy \cite{Lev:51} and Milman \cite{Mil:71,Mil:88}). Before \cite{GPW:09,Loe:13,GMS:15}, Sturm \cite{Stu:06a} set up a natural transport distance mimicking the $2$-Wasserstein metric. Later, inspired by work of Mémoli \cite{Mem:11} in theoretical computer science, he defined and studied distortion distances of normalized metric measure spaces \cite{Stu:12}. All prior notions of measured Gromov--Hausdorff convergence are known to be essentially equivalent.

For every $n\in\N\cup\{\infty\}$, let $\scrM_n$ be a normalized bounded Lorentzian metric measure space. We introduce the subsequent three notions of \emph{measured Lorentz--Gromov--Hausdorff convergence} of the sequence $(\scrM_n)_{n\in\N}$ to $\scrM_\infty$.
\begin{enumerate}[label=\arabic*.]
\item \textbf{Intrinsic convergence}, cf.~\cref{Def:Intrinsic}.\\We say $(\scrM_n)_{n\in\N}$ converges intrinsically to $\scrM_\infty$ if the matrix laws $\smash{(\bar{\meas}_n^k)_{n\in\N}}$ converge narrowly to $\smash{\bar{\meas}_\infty^k}$ from \eqref{Eq:barm^k} in the space  of Borel probability measures on $\smash{\R^{k\times k}}$ for every cardinality $k\in\N$. This  is   motivated from our Gromov reconstruction \cref{Th:Gromov reconstruction}. It reflects Gromov's corresponding  approach  \cite{Gro:99} from metric measure geometry.
\item \textbf{Distortion convergence}, cf.~\cref{Def:Distortion convergence}.\\ 
Another notion of convergence asks for
\begin{align*}
\lim_{n\to\infty} \LL\bdDelta_0(\scrM_n,\scrM_\infty)=0.
\end{align*}
Here, $\LL\bdDelta_0$ means the \emph{distortion distance}  defined by
\begin{align*}
\LL\bdDelta_0(\scrM,\scrM') = \inf\inf\!\big\lbrace \varepsilon > 0:\PPP\big[\big\vert\tau(X,Y)-\tau'(X',Y')\big\vert > \varepsilon\big] \leq \varepsilon\big\rbrace;
\end{align*}
the outer infimum is taken over all $\mms$-valued random variables $X$ and $Y$ and all $\mms'$-valued random variables $X'$ and $Y'$ modeled on a common probability space $(\Omega,\scrF,\PPP)$ such that the random variables $(X,X')$ and $(Y,Y')$ are independent. The quantity $\LL\bdDelta_0(\scrM,\scrM')$ mimics the Fan metric \eqref{Eq:Fan} metrizing narrow convergence of probability measures on Polish spaces. It is a Lorentzian analog of Greven--Pfaffelhuber--Winter's \emph{Euran\-dom distance} \cite{GPW:09} (called \emph{$L^0$-distortion distance} in Sturm \cite{Stu:12}). The Eurandom distance is incomplete, yet its completion has been described by Sturm \cite{Stu:12}. Among our three proposals, it has the closest link to the unmeasured Lorentz--Gromov--Hausdorff convergence of Müller \cite{Mue:22} and Minguzzi--Suhr \cite{MS:22b}. $L^p$-versions of this distance, where $p\in[1,\infty]$ ---  along with links of our definitions to  \cite{Mue:22,MS:22b} ---  are  discussed in \cref{Sub:Strong}. 
\item \textbf{Box convergence}, cf.~\cref{DeF:BoxConv}.\\ This convergence is induced by a Lorentzian variant of Gromov's famous box distance \cite{Gro:99}, cf.~\eqref{Eq:squareuv} and \cref{Def:LGboxdef}. Unlike the above Eurandom distance, Gromov's  box distance is known to be complete.
\end{enumerate}
More precisely, all these notions define the convergence of \emph{isomorphism classes} of normalized bounded Lorentzian metric measure spaces. 

Fundamental properties of the above approaches to measured Lorentz--Gromov--Hausdorff convergence are established in  \cref{Sub:mLGH}. Notably, we show the distances from 2. and 3. are in fact metrics in \cref{Th:Metric prop,Th:MetricLBox}.  Moreover, we show the following interrelation, which combines  \cref{Pr:IntrDist,Th:Comparison}.

\begin{theorem}[Hierarchy]\label{Th:Hier} Box convergence implies distortion convergence. 

In turn, distortion convergence implies intrinsic convergence.
\end{theorem}

We believe these three notions are in fact equivalent. One possible way to show this would be to mimic Greven--Pfaffelhuber--Winter \cite{GPW:09} and Löhr \cite{Loe:13}. However, the relevant results in these works pass through a transport distance. Although one could define a Lorentzian analog of it using the distinction metric \eqref{Eq:Distinctionmetric}, the induced distance appears too strong and nonlocal, as already realized by Minguzzi--Suhr \cite{MS:22b}. Hence, this path is not a natural choice any more. Moreover, the arguments of Greven--Pfaffelhuber--Winter \cite{GPW:09} rely on deep covering properties of metric spaces; in Lorentzian signature, these would have to be developed and understood first. The work of Mondino--Sämann \cite{MS:25} appears relevant in this direction.

\subsection{Spacetime reconstruction}\label{Sub:SptReconst} Lastly, we discuss  \cref{Th:Gromov reconstruction} in light of the spacetime reconstruction problem. This expression  broadly refers to the challenge of recovering the structure of space\-time --- including its geometry, topology, and causality --- from more fundamental, often non-spatiotemporal, data or entities. Notable instances of it  arise in causal set theory (CST), quantum infor\-mation \cite{Swi:18}, or AdS/CFT correspondence \cite{DHSS:01}.

We illustrate the perspective of CST by way of example, referring the interested reader to the overviews of  Surya \cite{Sur:19} or Dowker--Surya \cite{DS:24} for details. CST is an approach to quantum gravity which proposes spacetime is essentially discrete at microscopic scales. The continuum emerges at macroscopic scale from order and volume, two data inherent to our setting. In CST, the spacetime reconstruction problem in fact translates into its fundamental conjecture, the \emph{Hauptvermutung} stipulated  by Bombelli--Lee--Meyer--Sorkin \cite{BLMS:87}. Its resolution would justify the CST approach to quantum gravity. Roughly speaking, it asserts if a given finite poset embeds ``faithfully'' above a suitable scale into two finite volume spacetimes, the latter two should be ``close''  at similar scales.

To date, even a rigorous mathematical formulation of the Hauptvermutung is missing, despite some attempts by Bombelli \cite{Bom:00}, Noldus \cite{Nol:04}, Bombelli--Noldus \cite{BN:04}, Bombelli--Noldus--Tafoya \cite{BNT:12},  Müller \cite{Mue:25}, and Mondino--Sämann \cite{MS:25} --- all by Gromov--Hausdorff theory.  One possible access to the Hauptvermutung could thus be to probe \cref{Th:Gromov reconstruction} to the statistical spacetime  reconstruction  by Poisson point processes suggested by Bombelli \cite{Bom:00}. We note the relevance of \cref{Th:LG prob} here: if $N \sim \mathrm{Poi}(1)$ is independent of $\smash{(X_i)_{i\in\N}}$, then $\smash{\mu := \delta_{X_1} + \dots +\delta_{X_N}}$ is a Poisson point process with intensity measure $\meas$. Some care is required, however. First, \cref{Th:Gromov reconstruction} is a rigid statement about isomorphy. It would be interesting to develop an ``almost-isomorphy'' version of it, which would reflect the  quantification of ``closeness'' in the Hauptvermutung. Alternatively, the latter could be made concrete by our distortion or box distance. Second, our concept of isomorphy forces the time separation functions in question (and not only the orders) to coincide. If we fix a finite poset with a time separation function \emph{a priori}, the probability that the  random matrices from \cref{Th:LG prob} attain exactly this set is  zero.

As a first step, we will consider the following rigid version of the Hauptvermutung known as Bombelli's conjecture \cite{Bom:00}. 

\begin{conjecture}[Bombelli's conjecture]\label{Conj:Bomb}  Assume that given two distinguishing\footnote{A spacetime is \emph{distinguishing} if for every points $x$ and $y$ in it,  $I^+(x)=I^+(y)$ or $I^-(x)=I^-(y)$ implies $x=y$.} finite volume spacetimes, for every finite poset $C$ the probabilities that samples of Poisson point processes from   either spacetime are order isomorphic to $C$ coincide. Then the two spacetimes in question are isometric.
\end{conjecture}

Observe that for every $k\in\N$, there are only finitely many posets of cardinality $k$ up to order isomorphy. This rules out the second concern above. Due to the flexibility of the proof of our Gromov reconstruction \cref{Th:Gromov reconstruction}, we believe a version of \cref{Conj:Bomb} is within reach with similar methods as those developed in this paper. This will be addressed in a forthcoming paper.

\subsection{Organization} In \cref{Sub:BLMSsumm}, we collect basics of probability theory and Minguzzi--Suhr's bounded Lorentzian metric spaces \cite{MS:22b}. This is followed by our introduction and study of isomorphy of bounded Lorentzian metric measure spaces. Lastly, we review parametrizations of measure spaces, a concept required to set up the Lorentz--Gromov box distance.

In \cref{Sub:GromovRecon}, we prove our Gromov reconstruction  \cref{Th:Gromov reconstruction}.

In the first three parts of \cref{Sub:mLGH}, we introduce our  definitions of measured Lorentz--Gromov--Hausdorff convergence outlined in \cref{Sub:MLGHCVG} and establish their fundamental properties. The last part addresses \cref{Th:Hier}.

Finally, \cref{Sub:Strong} contains further distances and their properties which are similar to our distortion distance. In addition, it briefly compares our notions of measured Lorentz--Gromov--Hausdorff convergence with the unmeasured notions of Müller \cite{Mue:22} and Minguzzi--Suhr \cite{MS:22b}.

\subsection{Acknowledgments} MB acknowledges financial support from the EPFL through a Bernoulli Instructorship. His research was supported in part by the Fields Institute for Research in Mathematical Sciences. We warmly thank Fay Dowker for pointing out Bombelli's conjecture to us. We thank the organizers of the  ``Thematic Program on Nonsmooth Riemannian and Lorentzian Geometry'' (Fields Institute, 2022), the workshop ``Non-regular Spacetime Geometry'' (Erwin Schrödinger International Institute for Mathematics and Physics, 2023), and the ``Kick-off Workshop `A new geometry for Einstein's theory of relativity and beyond' \!\!'' (University of Vienna, 2025), during which parts of this work were discussed, for creating a stimulating research atmosphere. This research was funded in part by the Austrian Science Fund (FWF) [Grants DOI \href{https://doi.org/10.55776/STA32}{10.55776/STA32} and \href{https://doi.org/10.55776/EFP6}{10.55776/EFP6}]. For open access purposes, the authors have applied a CC BY public copyright license to any author accepted manuscript version arising from this submission.

\addtocontents{toc}{\protect\setcounter{tocdepth}{2}}

\section{Bounded Lorentzian metric measure spaces}\label{Sub:BLMSsumm}

In  this part, we introduce the structural notion of bounded Lorentzian metric measure spaces by adding a reference measure to the bounded Lorentzian metric spaces of Minguzzi--Suhr \cite{MS:22b}. Moreover, we will define a natural notion of their isomorphy and discuss natural properties.

\subsection{Convergence of probability measures}\label{Sub:Cvg prob meas} We first collect some repeatedly used notions of probability theory. We refer to the monographs of Billingsley \cite{Bil:99}, Ambrosio--Gigli--Savaré \cite{AGS:08}, and Villani \cite{Vil:09} for details.

Throughout this discussion, let $\sfX$ and $\sfY$ be Polish spaces, i.e., separable topological spaces that are completely metrizable.

The space of Borel probability measures on $\sfX$ will be  denoted by $\Prob(\sfX)$. Given a Borel measurable map $T\colon\sfX\to\sfY$, the \emph{push-forward} $T_\push\mu\in\Prob(\sfY)$ of $\mu$ under $T$ is defined by $\smash{T_\push\mu[E] := \mu[T^{-1}(E)]}$ for every Borel measurable set $E\subset\sfY$. Given any $\mu\in\Prob(\sfX)$ and any $\nu\in\Prob(\sfY)$, a \emph{coupling} of $\mu$ and $\nu$ is an element $\pi\in \Prob(\sfX\times\sfY)$ with  $(\pr_1)_\push\pi = \mu$ and $(\pr_2)_\push\pi = \nu$, where here and henceforth, $\pr_i$ denotes the projection map onto the $i$-th coordinate; in this case, $\mu$ and $\nu$ are called \emph{marginals} of $\pi$. Heuristically, we think of a coupling as ``matching'' the mass distributions of its marginals: the infinitesimal portion $\rmd\pi(x,y)$ matches the infinitesimal portions $\rmd\mu(x)$ around $x\in\sfX$ and $\rmd\nu(y)$ around $y\in\sfY$. For instance, the product measure $\mu\otimes \nu$ is always a coupling of $\mu$ and $\nu$ (which matches each point in $\supp\mu$ with each point in $\supp\nu$). If $\nu = T_\push\mu$ for a Borel measurable map  $T\colon\sfX\to\sfY$ as above, then $(\Id,T)_\push\mu$ is a coupling of $\mu$ and $\nu$ which is concentrated on the graph of $T$ --- in particular, it differs from the product measure $\mu\otimes \nu$.

The \emph{support} of $\mu\in\Prob(\mms)$, denoted  $\supp\mu$, is the largest  closed subset of $\mms$ such that  every neighborhood of any of its points has positive $\mu$-measure; explicitly,
\begin{align*}
\supp\mu := \{x\in\mms : \mu[U]>0\textnormal{ for every open set }U\subset \mms \textnormal{ containing }x\}.
\end{align*}
Since $\mu$ constitutes a Radon measure on a Hausdorff topological space, every Borel set $N\subset \mms \setminus \supp\mu$ satisfies $\mu[N]= 0$. In particular, this yields $\mu[\supp\mu]=1$.

We say a sequence $(\mu_n)_{n\in\N}$ in $\Prob(\sfX)$ \emph{converges narrowly} to $\mu\in\Prob(\sfX)$ if for every bounded and continuous function $\varphi\colon \sfX\to\R$, written $\smash{\varphi\in\Cont_\bounded(\sfX)}$,
\begin{align}\label{Eq:Narrow}
\lim_{n\to\infty}\int_\sfX \varphi\d\mu_n =\int_\sfX \varphi\d\mu.
\end{align}
This is equivalent to any of the following conditions.
\begin{itemize}
\item Upper semicontinuity on closed sets, i.e.~for every closed $C\subset\sfX$,
\begin{align*}
\limsup_{n\to\infty}\mu_n[C] \leq \mu[C].
\end{align*}
\item Lower semicontinuity on open sets, i.e.~for every open $U\subset \sfX$,
\begin{align*}
\mu[U] \leq \liminf_{n\to\infty} \mu_n[U].
\end{align*}
\end{itemize}
In fact, the narrow topology is metrizable, cf.~e.g.~Billingsley \cite{Bil:99}*{Thms.~6.8, 6.9};  fixing a complete and separable metric $\met$ inducing the topology of $\sfX$, a metric on $\Prob(\sfX)$ certifying this property is the so-called \emph{Fan metric} $\sfd_{\textnormal{Fan}}$ given by
\begin{align}\label{Eq:Fan}
\sfd_{\textnormal{Fan}}(\mu,\nu) := \inf\inf\!\big\lbrace \varepsilon > 0 : \pi\big[\big\{(x,y) \in \sfX^2:  \met(x,y) > \varepsilon\big\}\big] \leq \varepsilon\big\rbrace,
\end{align}
where the outer infimum is taken over all couplings $\pi$ of $\mu$ and $\nu$.

\begin{remark}[Testing narrow convergence against a dense class of functions]\label{Re:Testfcts} Let $\scrF$ constitute a class of bounded and continuous functions which is uniformly dense in the space of bounded and continuous functions on $\sfX$. Then a sequence $(\mu_n)_{n\in\N}$ in $\Prob(\sfX)$ converges narrowly to $\mu\in\Prob(\sfX)$ if and only if \eqref{Eq:Fan} holds for every $\varphi\in\scrF$.
\end{remark}

We frequently use the subsequent features of narrow convergence.
\begin{itemize}
\item Two sets $\scrC\subset \Prob(\sfX)$ and $\scrC'\subset\Prob(\sfY)$ are narrowly precompact if and only if the set of all couplings of elements in $\scrC$ and $\scrC'$ is narrowly precompact in $\Prob(\sfX\times\sfY)$, cf.~e.g.~Villani \cite{Vil:09}*{Lem.~4.4}.
\item Let $(\mu_n)_{n\in\N}$ be a sequence in $\Prob(\sfX)$ and $(\nu_n)_{n\in\N}$ be a sequence in $\Prob(\sfY)$.  Additionally, let $\mu\in\Prob(\sfX)$ and $\nu\in\Prob(\sfY)$. Then $(\mu_n \otimes \nu_n)_{n\in\N}$ converges narrowly to $\mu\otimes \nu$ if and only if $(\mu_n)_{n\in\N}$ converges narrowly to $\mu$ and $(\nu_n)_{n\in\N}$ converges narrowly to $\nu$, cf.~e.g.~Billingsley \cite{Bil:99}*{Thm.~2.8}.
\end{itemize}
Moreover, we have the following well-known  characterization of narrow precompactness,  cf.~e.g.~Billingsley \cite{Bil:99}*{Thms.~5.1, 5.2}. Let us call $\scrC\subset\Prob(\sfX)$  \emph{tight} if for every $\varepsilon >0$, there exists a compact subset $C\subset\sfX$ such that
\begin{align*}
\sup\!\big\lbrace \mu[\sfX\setminus C] : \mu \in\scrC\big\rbrace \leq \varepsilon.
\end{align*}

\begin{theorem}[Prohkorov's theorem] A set  $\scrC\subset\Prob(\sfX)$ is narrowly precompact if and only if it is tight. 
\end{theorem}

In particular, singletons are tight.

We will use the disintegration theorem in the following form, cf.~e.g.~Ambrosio--Gigli--Savaré \cite{AGS:08}*{Thm.~5.3.1} and its references.

\begin{theorem}[Disintegration theorem]\label{Th:DisintTheorem} Let $\mu\in \Prob(\sfX)$ and let  $T\colon \sfX\to \sfY$ be Borel measurable. We define $\nu \in\Prob(\sfY)$ by  $\nu:= T_\push\mu$. Then there exists a Borel measurable map $\mu_\cdot \colon \sfY \to \Prob(\sfX)$  such that
\begin{enumerate}[label=\textnormal{(\roman*)}]
\item for $\nu$-a.e.~$y\in \sfY$, the measure $\mu_y$ is concentrated on $T^{-1}(y)$ and
\item for every Borel measurable function $f\colon \sfX \to \R_+$,
\begin{align}\label{Eq:Disintform}
\int_\sfX f(x) \d\mu(x) = \int_\sfY\int_\sfX  f(x) \d\mu_y(x) \d\nu(y).
\end{align}
\end{enumerate}

Moreover, the map $\mu_\cdot$ is $\nu$-a.e.~uniquely determined by these two properties.
\end{theorem}

As a shorthand for the property \eqref{Eq:Disintform}, we will  write
\begin{align*}
\rmd\mu(x) = \int_\sfY \rmd\mu_y(x) \d\nu(y).
\end{align*}
The measure $\mu_y$ is called \emph{conditional law} of $\mu$ given $\{T=y\}$, where $y\in\sfY$.

One application of the disintegration theorem is the following well-known result, cf.~e.g.~Ambrosio--Gigli--Savaré \cite[Lem.~5.3.2]{AGS:08} for a proof. 

\begin{lemma}[Gluing lemma]\label{Le:Gluing} Given   $\smash{n\in\N}$ with $n\geq 2$, let $\sfX_1,\dots,\sfX_n$ be Polish spaces, respectively  endowed with Borel probability measures $\mu_1,\dots,\mu_n$. Moreover, let $\pi_{i,i+1}$ be a coupling of $\mu_i$ and $\mu_{i+1}$ for every $i\in\{1,\dots,n-1\}$. Then there exists $\alpha\in\Prob(\sfX_1\times\dots\times\sfX_n)$ such that for every $i\in \{1,\dots,n-1\}$,
\begin{align*}
(\pr_i,\pr_{i+1})_\push\alpha = \pi_{i,i+1}.
\end{align*}
In particular, $(\pr_1,\pr_n)_\push\alpha$ is a coupling of $\mu_1$ and $\mu_n$.
\end{lemma}

\subsection{Bounded Lorentzian metric spaces}\label{Sub:BLMSSSSSSs} We now  recall basics about bounded Lorentz\-ian metric spaces as defined by Minguzzi--Suhr \cite{MS:22b}. Notably, we review several (especially topological) properties that are repeatedly used below.

\begin{definition}[Bounded Lorentzian metric space  {\cite[Def.~1.1]{MS:22b}}]\label{Def:Bounded Lorentzian metric spaces} A \emph{bounded Lorentzian metric space} is a tuple $(\mms,\tsep)$ consisting of a nonempty set $\mms$ together with a function $\smash{\tsep\colon\mms^2\to \R_+}$ with the following properties.
\begin{enumerate}[label=\textnormal{\alph*\textcolor{black}{.}}]
\item\label{La:a} \textnormal{\textbf{Reverse triangle inequality.}} For every triple $x,y,z\in\mms$ with $\tsep(x,y)> 0$ and $\tsep(y,z)>0$, 
\begin{align*}
\tsep(x,y) + \tsep(y,z) \leq \tsep(x,z).
\end{align*}
\item\label{La:b} \textnormal{\textbf{Compact superlevel sets.}} There exists a topology $T$ on $\mms$ such that with respect to the product topology $T^2$, $\tsep$ is continuous and the superlevel set $\{\tsep\geq \varepsilon\}$ is compact for every $\varepsilon >0$.
\item\label{La:c} \textnormal{\textbf{Point distinction property.}} For every distinct points $x,y\in \mms$ there is  $z\in \mms$ such that $\tsep(x,z) \neq \tsep(y,z)$ or $\tsep(z,x) \neq \tsep(z,y)$.
\end{enumerate}
\end{definition}

We call $\tau$ \emph{time separation function}.

A strict subset of such a space $(\mms,\tsep)$, endowed with the appropriate restrictions of $\tau$ and $T$, is not a bounded Lorentzian metric space in general, as it may lack the point distinction property. This is where the subsequent  natural  construction becomes helpful. It will be  repeatedly used in this paper.

\begin{remark}[Distance quotient]\label{Re:Distquot} Assume $\mms$ is a space endowed with a topology $T$ and a $T^2$-continuous function $\smash{\tau\colon \mms^2\to \R_+}$. Define an equivalence relation $\sim$ on $\mms$ by $x\sim y$ provided $\tau(x,z) = \tau(y,z)$ and $\tau(z,x)=\tau(z,y)$ for every $z\in\mms$. On the quotient $\smash{\tilde{\mms}}$, we define a function $\smash{\tilde{\tau}\colon\tilde{\mms}^2 \to \R_+}$ by $\smash{\tilde{\tau}(q(x),q(y)) := \tau(x,y)}$ for every $x,y\in\mms$; here, $\smash{q\colon\mms\to\tilde{\mms}}$ means the canonical projection. If $(\mms,\tau)$ satisfies the properties  \ref{La:a} and \ref{La:b} from \cref{Def:Bounded Lorentzian metric spaces},  then $\smash{(\tilde{\mms},\tilde{\tau})}$ becomes a bounded Lorentzian metric space; cf.~Minguzzi--Suhr \cite{MS:22b}*{Prop.~1.19}.
\end{remark}

We write $x\ll y$ if $\tsep(x,y) > 0$ for $x,y\in\mms$. In this case, we say that $x$ lies in the \emph{chronological past} of $y$ and denote this symbolically by $x\in I^-(y)$, that $y$ lies in the \emph{chronological future} of $x$ and denote this by $y\in I^+(x)$, and that, more roughly, $x$ and $y$ are \emph{chronologically related}. Elementary adjacent  consequences are the following, summarized in Minguzzi--Suhr  \cite[Rem.~1.2]{MS:22b}.
%The corresponding \emph{chronological past}, \emph{chronological future}, and \emph{chronological diamond} are defined by
%\begin{align*}
%I^-(x) &:= \{x\in \mms : x\ll y\},\\
%I^+(y) &:= \{y\in\mms : x\ll y\},\\
%I(x,y) &:= I^+(x) \cap I^-(y).
%\end{align*}
The reverse triangle inequality implies $\ll$ is a transitive relation, yet it is irreflexive and not symmetric.  Continuity of $\tsep$ yields openness of chronological pasts and futures $I^\pm(x)$, where $x\in M$, and openness of $\ll$ in $\mms^2$. The time separation function $\tsep$ is  bounded on $\mms$, which justifies the terminology from \cref{Def:Bounded Lorentzian metric spaces}. In turn, this property combines with the reverse triangle inequality  to imply $\tsep(x,x)=0$ for all $x\in M$.

Moreover, there exists at most one point in $\mms$, denoted by $\smash{i^0}$, called \emph{spacelike boundary} of $\mms$, such that $\smash{\tsep(i^0,x) = \tsep(x,i^0) = 0}$ for every $x\in \mms$. 

In the sequel, all topological properties will be understood with respect to the coarsest topology $\scrT$ satisfying \ref{La:b} in \cref{Def:Bounded Lorentzian metric spaces}. As shown by Minguzzi--Suhr, this topology is in fact Polish \cite[Thm.~1.9]{MS:22b}, and if $\smash{i^0\in\mms}$ then $(\mms,\scrT)$ is compact \cite[Cor.~1.6]{MS:22b} (otherwise it is locally compact and $\sigma$-compact \cite[Prop.~1.3]{MS:22b}). By continuity of $\tsep$, it follows that the point $z$ in item \ref{La:c} can be found in an arbitrarily fixed countable and dense subset of $\mms$. In fact, all topologies satisfying \ref{La:b} relative to $\smash{\mms^\circ := \mms\setminus \{i^0\}}$ coincide with the topology $\scrT^\circ$ induced by $\scrT$ \cite[Cor.~1.6]{MS:22b}. In particular, the above constructed coarsest topology relative to $\mms^\circ$ coincides with $\scrT^\circ$. Thus, apart from neighborhood systems of the spacelike boundary of $\mms$, no ambiguity can arise.

\begin{remark}[One-point compactification]\label{Re:One-point} If the set $\mms$ does not contain a spacelike boundary, adjoining an abstract element $\smash{i^0}$ to $\mms$ by extending $\tau$ trivially by zero to $\smash{(\{i^0\} \times \bar\mms)\cup(\bar\mms\times\{i^0\})}$ defines a compact bounded Lorentzian metric space $\smash{(\tilde\mms,\tilde\tsep)}$. The coarsest topology $\tilde\scrT$ relative to $\tilde\mms$ according to the above paragraph coincides with the one-point compactification of $\scrT$, cf.~Minguzzi--Suhr \cite[Prop.~1.7]{MS:22b}.
\end{remark}

As discussed by Minguzzi--Suhr  \cite[Thm.~2.5]{MS:22b}, certain subsets of globally hyperbolic spacetimes such as causal diamonds with removed spacelike boundary are bounded Lorentzian metric spaces. 

Another class of examples fitting into the framework of \cref{Def:Bounded Lorentzian metric spaces},  cf.~\cite[Prop.~2.2]{MS:22b}, are called \emph{causets} \cite{MS:22b}. These are more general then the causal sets from CST.

\begin{definition}[Causet]\label{Def:Causet} Let $\sfS$ be a finite set with a map $\smash{\tsep\colon \sfS^2 \to\R_+}$ obeying points \ref{La:a} and \ref{La:c} of \cref{Def:Bounded Lorentzian metric spaces}. Then $(\sfS,\tsep)$ is called \emph{causet}.
\end{definition}

Causets can easily be identified with distance matrices by ``forcing the point distinction property'' according to the following simple \cref{Le:Causetdist}. Given any $k\in\N$, recall the definition \eqref{Eq:Gkdef} of $\G^k\subset\R^{k\times k}$. Moreover, we define
\begin{align*}
\S^k&:= \big\{a \in\G^k : \textnormal{for every distinct }i,j\in\{1,\dots,k\} \textnormal{ there }\\
&\qquad\qquad \textnormal{exists }l\in\{1,\dots,k\}  \textnormal{ such that }a_{il}\neq a_{jl} \textnormal{ or }a_{li} \neq a_{lj}\big\}.
\end{align*}

\begin{lemma}[Causets vs.~distance matrices]\label{Le:Causetdist} Define $\sfS := \{x_1,\dots,x_k\}$, where $k\in\N$ and $x_i\neq x_j$ for every $i,j\in\{1,\dots,k\}$. Given $\smash{a\in \S^k}$, define $\smash{\tau\colon \sfS^2\to\R_+}$ by $\smash{\tau(x_i,x_j) := a_{ij}}$. Then $\smash{(\sfS,\tau)}$ is a causet according to \cref{Def:Causet} \textnormal{(}with the discrete topology\textnormal{)}.

Conversely, assume $\smash{\tau\colon\sfS^2\to\R_+}$ is such that $(\sfS,\tau)$ is a causet. Furthermore, define $\smash{a\in\R^{k\times k}}$ by $a_{ij} := \tau(x_i,x_j)$. Then $\smash{a\in \S^k}$.
\end{lemma}

Before turning to a further example,  recall the \emph{distinction metric} $\smash{\sfn\colon\mms^2\to \R_+}$ induced by $\tsep$ defined by Minguzzi--Suhr  \cite[§4.3]{MS:22b}, where
\begin{align}\label{Eq:Distinctionmetric}
\sfn(x,y) = \max\!\Big\lbrace\! \sup_{z\in \mms} \big\vert\tsep(x,z) - \tsep(y,z)\big\vert,\sup_{z\in\mms} \big\vert\tsep(z,x) - \tsep(z,y) \big\vert\Big\rbrace.
\end{align}
By \cite[Prop.~4.19]{MS:22b} it coincides with the \emph{Noldus metric} introduced by Noldus \cite{Nol:04}. It is trivial that $\tsep$ is $1$-Lipschitz continuous with respect to  the canonical product metric $\smash{\sfn^{\oplus 2}}$ on $\smash{\mms^2}$ \cite[Prop.~4.20]{MS:22b}. Moreover, the topology induced by $\sfn$ coincides with $\scrT$ \cite[Prop.~4.21]{MS:22b} --- in particular, $\sfn$ is $\scrT$-continuous.

\begin{example}[Disjoint union]\label{Ex:Disjoint union} Let $(\mms,\tsep)$ and $(\mms',\tsep')$ be two bounded Lorentzian metric spaces with spacelike boundaries $\smash{i^0}$ and $\smash{i'^0}$, respectively. Let $\smash{\hat\mms}$ constitute  the disjoint union of $\mms$ and $\mms'$ modulo identification of $\smash{i^0}$ and $\smash{i'^0}$ to a single point $\hat\imath\in \hat\mms$, with induced topology. Define $\smash{\hat\tsep\colon\hat\mms^2\to [0,\infty)}$ by
\begin{align*}
\hat\tsep(x,y) := \begin{cases} \tsep(x,y) & \textnormal{if } x,y\in \mms,\\
\tsep'(x,y) & \textnormal{if }x,y\in\mms',\\
0 & \textnormal{otherwise}.
\end{cases}
\end{align*}
Then $\smash{(\hat\mms,\hat\tsep)}$ is a bounded Lorentzian metric space as well. Moreover, the induced distinction metric $\smash{\hat\sfn}$ on $\smash{\hat\mms}$ is a \emph{coupling} of the distinction metrics $\sfn$ on $\mms$ and $\sfn'$ on $\mms'$ after the definition of Sturm \cite[Def.~3.2]{Stu:06a}, namely it equals $\sfn$ on $\mms^2$ and $\sfn'$ on $\mms'^2$. This simply follows because the respective suprema in the definition of $\smash{\hat\sfn}$ do not see other points than in $\mms$ or $\mms'$, respectively, by definition of $\hat\tsep$.
\end{example}

\subsection{Distance-preserving maps and isometries}\label{Sub:Distpres} Let $(\mms,\tsep)$ and $(\mms',\tsep')$ be bounded Lorentzian metric spaces. A central role in our work will be played by the following objects introduced and studied by Minguzzi--Suhr \cite{MS:22b}.

\begin{definition}[Isometry {\cite[Def.~1.15]{MS:22b}}]\label{Def:Isometryy} A map $\iota\colon\mms\to\mms'$ is called \emph{distance-preserving} if for all $x,y\in\mms$, 
\begin{align*}
\tsep(x,y) = \tsep'(\iota(x),\iota(y)).
\end{align*}

If $\iota$ is additionally bijective, it is termed \emph{isometry}.
\end{definition}

Trivial examples of distance-preserving (but generally not bijective) maps are the canonical  distance quotient projection $q$ from \cref{Re:Distquot} and the canonical inclusion maps from \cref{Ex:Disjoint union}.

By the point distinction property, every distance-preserving map is injective, as shown by Minguzzi--Suhr \cite[Thm.~3.1]{MS:22b}. Every isometry is in fact a homeomorphism \cite[Thm.~3.4]{MS:22b}. A  somewhat more general property we  often use is  if $\iota\colon \mms \to \mms'$ and $\iota'\colon\mms'\to\mms$ are distance-preserving --- with no assumption on surjectivity and no a priori relation between these maps --- then $\iota$ and $\iota'$ are isometries \cite[Thm.~3.5]{MS:22b}, hence homeomorphisms. Applying this statement to $\mms^\circ$ and $\mms'^\circ$, respectively, yields that $\mms$ contains its spacelike boundary if and only if $\mms'$ does, in which case necessarily $\smash{\iota(i^0) = i'^0}$ and $\smash{\iota'(i'^0) = i^0}$ \cite[Rem.~3.6]{MS:22b}. Lastly, in the situation of \cref{Re:One-point}, every isometry $\iota\colon \mms\to \mms'$ extends to an isometry $\smash{\tilde\iota\colon\tilde\mms\to \tilde\mms'}$ by setting $\smash{\tilde\iota(i^0) := i'^0}$ \cite[Cor.~1.17]{MS:22b}.

\begin{remark}[Metric balls]\label{Re:Metricballs} Let $\iota\colon\mms\to \mms'$ be an isometry. Endow $\mms$ and $\mms'$ with the respective distinction metrics $\sfn$ and $\sfn'$. Then for every $x\in\mms$ and every $r>0$, the definition \eqref{Eq:Distinctionmetric} directly implies $\smash{B_r'(\iota(x)) = \iota(B_r(x))}$. 
\end{remark}

Although isometries enjoy good continuity properties, unlike the metric case it is a priori unclear whether a distance-preserving map, if defined on a dense subset, extends  to the entire space. This is the content of the following result. 

\begin{lemma}[Extension of distance-preserving maps]\label{Le:Extension} Let $(\mms,\tsep)$ and $(\mms',\tsep')$ be bounded Lorentzian metric spaces. Let $D $  be dense in $\mms$, and let $\iota\colon D \to \mms'$ be a given distance-preserving map whose image $\iota(D)$ is dense in $\mms'$ as well. Then $\iota$ extends uniquely to a nonrelabeled isometry  $\smash{\iota\colon \mms\setminus \{i^0\}\to \mms'\setminus\{i'^0\}}$ \textnormal{(}which, if desired, can be extended to the respective spacelike boundaries\textnormal{)}.
\end{lemma}

We note the topological property that will grant this extension is \emph{compactness} --- which is somewhat inherent in our framework by point \ref{La:b} from \cref{Def:Bounded Lorentzian metric spaces}. We believe that \cref{Le:Extension} should extend to a suitable notion of \emph{completeness} of the spaces in question, such as the one recently advertised by Gigli \cite{Gig:25} and \cite{BBC+:24}. 

\begin{proof}[Proof of \cref{Le:Extension}] We will first construct a distance-preserving extension of $\iota$. Its surjectivity will be shown at the end of the proof.

Let $\smash{x\in\mms\setminus \{i^0\}}$ and let $(x_i)_{i\in\N}$ be a sequence in $D$ which converges to $x$. Since $\smash{x\neq i^0}$ and owing to the density of $D$, there exists $\bar{x}\in D$ with $x\ll \bar{x}$ or $\bar{x}\ll x$. Without restriction, we may and will assume $\bar{x}\ll x$, otherwise the proof is similar. Then for every large enough $i\in\N$,
\begin{align}\label{Eq:iotabar}
2\,\tsep'(\iota(\bar{x}),\iota(x_i)) = 2\,\tsep(\bar{x},x_i) \geq \tsep(\bar{x},x) > 0.
\end{align}
Thus the sequence $(y_i)_{i\in\N}$ given by $y_i := \iota(x_i)$ takes values in a compact subset of $\mms'$. Hence, there exists a subsequence $\smash{(y_{i_k})_{k\in\N}}$ converging to some $y\in \mms'$. By \eqref{Eq:iotabar}, we clearly have $\smash{y\neq i'^0}$.  We then set $f(x) := y$. This procedure defines a map $\smash{f\colon \mms\setminus \{i^0\} \to \mms'\setminus\{i'^0\}}$. By continuity of $\tau$ and $\tau'$, $f$ is distance-preserving. We now show this  definition is independent of the chosen subsequence. By considering constant sequences in $D$, this will show $f$ coincides with $\iota$ on $D$ as a byproduct. Let  $\smash{(y_{i_l})_{l\in\N}}$ be another  subsequence with limit $\smash{y'\in\mms'}$ constructed as above. As before, the relation \eqref{Eq:iotabar} readily implies $\smash{y'\neq i'^0}$. We claim $y=y'$. Assume to the contrary this is not the case. Then since $\tau'$ distinguishes points and $\iota(D)$ is dense in $\mms'$, there exists $z\in D$ such that $\tsep'(y,\iota(z)) \neq \tsep'(y',\iota(z))$ or $\tsep'(\iota(z),y) \neq \tsep'(\iota(z),y')$.  Without loss of generality, we may and will assume $\tsep'(\iota(z),y) < \tsep'(\iota(z),y')$, as the other cases are treated analogously. Then
\begin{align*}
\tsep(z,x) &= \lim_{k\to\infty}  \tsep(z,x_{i_{k}})\\ 
&=  \lim_{k\to\infty}  \tsep'(\iota(z),y_{i_{k}})\\ 
&< \lim_{l\to\infty}  \tsep'(\iota(z),y_{i_{l}})\\
&= \lim_{l\to\infty} \tsep(z,x_{i_{l}})\\
&= \tsep(z,x).
\end{align*}
This leads to a contradiction.

Finally, we prove surjectivity. Let $\smash{y\in\mms'\setminus\{i'^0\}}$ be given. Since $\iota(D)$ is dense in $\mms'$, there exists a sequence $(x_i)_{i\in\N}$ in $\smash{\mms\setminus\{i^0\}}$ such that $(y_i)_{i\in\N}$ converges to $y$, where $y_i = \iota(x_i)$. Since $\smash{\mms\cup\{i^0\}}$ is compact, we may and will assume $(x_i)_{i\in\N}$ converges to a point $x\in\mms$ (which necessarily differs from $i^0$). Repeating the above argument using the point distinction property of $(\mms',\tau')$ establishes $y= f(x)$, to finish the proof.
\end{proof}

\subsection{Isomorphy and couplings} Gromov promoted the concept of metric measure spaces and their isomorphy --- which turned out to be the basis of modern metric measure geometry --- in his famous book \cite{Gro:99}. We propose Lorentzian analogs of these notions here. Technically, since the reference topology is automatically Polish by the results of Minguzzi--Suhr \cite{MS:22b} discussed in \cref{Sub:BLMSSSSSSs}, standard results of measure theory  as described in \cref{Sub:Cvg prob meas} apply without ambiguities.

Before proceeding, we point out the concept of Cavalletti--Mondino's measured Lorentzian pre-length spaces \cite{CM:20}*{Def.~1.11},  set up by adding a reference measure to Kunzinger--Sämann's Lorentzian pre-length spaces \cite{KS:18}.  It is related to our bounded Lorentzian metric measure spaces. Their notion requires the reference measure to be Radon and to have  full support. The first property is automatic in most of our paper, as we work with finite reference measures on Polish spaces, which are always Radon. On the other hand, we carefully avoid any ad hoc assumption about full support. Defining the relevant maps only on the respective supports  conflicts a bit with the lack of the point distinction property of general subsets of bounded Lorentzian metric spaces (cf.~the comment before \cref{Re:Distquot}). We will make up for this by incorporating distance quotients into our definition of isomorphy. This notion still behaves well under couplings, cf.~\cref{Le:Char iso}.

An alternative definition of measured Lorentz--Gromov--Hausdorff convergence via $\varepsilon$-nets of causal diamonds has recently been proposed by Mondino--Sämann \cite{MS:25}.

\begin{definition}[Bounded Lorentzian metric measure spaces] A \emph{bounded Lorentz\-ian metric measure space} is a triple $(\mms,\tau,\meas)$ consisting of a bounded Lorentzian metric space $(\mms,\tau)$ and a nontrivial and locally finite Borel measure $\meas$ on it.
\end{definition}

We often abbreviate such a triple $(\mms,\tau,\meas)$ by $\scrM$.

The central subclass for our paper is made of the following: we call a bounded Lorentzian metric measure space $\scrM$
\begin{itemize}
\item  \emph{finite} if $\meas[\mms] < \infty$ and
\item  \emph{normalized} if $\meas[\mms]=1$.
\end{itemize}

Given a space $\mms$ with a time separation function $\tau$ and a Borel measure $\meas$, let $\smash{(\tilde{\mms},\tilde{\tau})}$ designate the distance quotient of $(\supp\meas,\tau\rvert_{\supp\meas\times\supp\meas})$ from \cref{Re:Distquot}, with canonical projection $\smash{q\colon\supp\meas\to\tilde{\mms}}$.  Endowing $\smash{(\tilde{\mms},\tilde{\tau})}$ with the natural quotient measure $\smash{\tilde{\meas} := q_\push\meas}$ yields a bounded Lorentz\-ian metric measure space $\smash{\tilde{\scrM}}$ that --- with some abuse of terminology --- we call the \emph{distance quotient} of $\scrM$. In the following, we frequently use the fact $\smash{q(\supp\meas) = \supp\tilde{\meas}}$, making $\smash{\supp\tilde{\meas}}$ into a bounded Lorentzian metric space. Furthermore, with some abuse of notation we will continue writing $\smash{i^0}$ for the spacelike boundary of the distance quotient.

\begin{definition}[Isomorphy]\label{Def:Isomorphy2} Two bounded Lorentzian metric measure spaces $\scrM$ and $\scrM'$  are termed  \emph{isomorphic} if there exists a distance-pre\-serving Borel measurable map $\iota \colon \supp q_\push\meas\cup\{i^0\} \to \supp q'_\push\meas'\cup\{i'^0\}$, called \emph{isomorphism}, such that  
\begin{align*}
\iota_\push(q_\push\meas) = q'_\push\meas'.    
\end{align*}
\end{definition}

We collect some observations about \cref{Def:Isomorphy2}. 

\begin{remark}[Inclusion of spacelike boundaries]\label{Re:Since iso} As isomorphisms are distance-preserving and spacelike boundaries are unique, in the context of \cref{Def:Isomorphy2} we have $\smash{i^0\in \supp q_\push\meas}$ if and only if $\smash{i'^0\in\supp q'_\push\meas'}$. In other words, if both supports do not contain the respective spacelike boundaries, we may and will always define $\iota(i^0) := i'^0$ without destroying the properties of an isomorphism.

We chose to include $i^0$ and $i'^0$ in the definition of  isomorphisms to ensure that addition or removal of spacelike boundaries does not alter isomorphy.
\end{remark}

\begin{remark}[About injectivity] The only reason for the distance quotients in the above definition is that distance-preservation of isomorphisms implies their injectivity, cf.~Minguzzi--Suhr \cite{MS:22b}*{Thm.~3.1}. Otherwise, injectivity would have to be assumed in addition. On the other hand, distance quotients are distance-preserving. Thus, whenever we find a distance-preserving map $\smash{\iota\colon\supp\meas\cup\{i^0\} \to \supp\meas'\cup\{i'^0\}}$ such that $\smash{\iota_\push\meas = \meas'}$, its natural quotient map is an isomorphism.
\end{remark}

\begin{remark}[Good representative]\label{Re:Goodrepr} If  $\supp\meas$ and $\supp\meas'$  already constitute bounded Lorentzian metric spaces, \cref{Def:Isomorphy2} reduces to the existence of a distance-preserving Borel measurable map $\iota\colon\supp\meas\cup\{i^0\}\to\supp\meas'\cup\{i'^0\}$ --- which is in particular injective --- with the property  $\smash{\iota_\push\meas = \meas'}$. 
In particular, this applies if $\meas$ and $\meas'$ have full support.
\end{remark}

\begin{proposition}[Homeomorphy]\label{Pr:Homeomorphy} Let $\smash{\iota\colon\supp q_\push\meas\cup\{i^0\} \to \supp q'_\push\meas'\cup\{i'^0\}}$ be an isomorphism of bounded Lorentzian metric measure spaces $\scrM$ and $\scrM'$ according to \cref{Def:Isomorphy2}. Then $\iota$ is an isometry, hence a homeomorphism.
\end{proposition}

In particular, isomorphic bounded Lorentzian metric measure spaces have the same total mass.

\begin{proof}[Proof of \cref{Pr:Homeomorphy}] Recall $\iota$ is already known to be injective. As $\smash{\iota(\supp q_\push\meas)}$ has full $\smash{q'_\push\meas'}$-measure, this set is dense in $\smash{\supp q'_\push\meas'}$. Therefore, \cref{Le:Extension} already forces $\iota$ to be an isometry, hence a homeomorphism.
\end{proof}

Consequently, isomorphy induces an equivalence relation on the totality of all bounded Lorentzian metric measure spaces. The equivalence class of $\scrM$ will be denoted by $[\scrM]$. Given any $\smash{m\in\N\cup\{\infty\}}$, $\smash{\MM_m}$ designates the set of all equivalence classes of bounded Lorentzian metric measure spaces with total mass $m$.

By \cref{Re:Goodrepr} above, every bounded Lorentzian metric measure space $\scrM$ is isomorphic to a bounded Lorentzian metric measure space with full support, namely  the distance quotient of $\supp\meas$ endowed with the induced quotient measure. Any such space will be called \emph{good representative} of $\scrM$.

Our notions of convergence will only depend on the isomorphism classes of the normalized bounded Lorentzian metric measure spaces in question. To this aim, we now characterize isomorphy by the existence of appropriate couplings of the reference measures (which, notably, do not involve the quotient maps $q$ and $q'$). 

\begin{lemma}[Coupling characterization of isomorphy]\label{Le:Char iso} The following claims are equivalent for two normalized bounded Lorentzian metric measure spaces $\scrM$ and $\scrM'$. 
\begin{enumerate}[label=\textnormal{\textcolor{black}{(}\roman*\textcolor{black}{)}}]
\item\label{La:111} The spaces $\scrM$ and $\scrM'$ are isomorphic.
\item\label{La:222} There exists a coupling  $\pi$ of $\meas$ and $\meas'$ such that $\pi^{\otimes 2}$-a.e.~point $(x,x',y,y')\in (\mms\times\mms')^2$ has the property
\begin{align}\label{Eq:Gleichheit}
\tsep(x,y) = \tsep'(x',y').
\end{align}
\item\label{La:333} There exist a normalized bounded Lorentzian metric measure space $\smash{\hat\scrM=(\hat\mms, \hat\tau,\hat\meas)}$ with spacelike boundary $\hat{\imath}^0$ and $\smash{\supp\hat\meas \cup \{\hat{\imath}^0\} = \hat\mms}$ as well as distance-pre\-serving maps $\smash{\hat{\iota}\colon \hat\mms \to \supp q_\push\meas\cup\{i^0\}}$ and $\smash{\hat{\iota}'\colon \hat\mms \to \supp q'_\push\meas'\cup\{i'^0\}}$ with 
\begin{align*}
\hat{\iota}_\push\hat\meas &= q_\push\meas,\\
\hat{\iota}'_\push\hat\meas &= q'_\push\meas'.
\end{align*}
\end{enumerate}
\end{lemma}

\begin{proof} \ref{La:111} $\Longrightarrow$ \ref{La:222}.  Let $\smash{\iota\colon \supp q_\push\meas\cup\{i^0\} \to \supp q'_\push\meas'\cup\{i'^0\}}$ realize the isomorphy of $\scrM$ and $\scrM'$. Thus, $\smash{\tilde{\pi} := (\Id,\iota)_\push\tilde{\meas}}$ is a coupling of $\smash{\tilde{\meas} := q_\push\meas}$ and $\smash{\tilde{\meas}' := q_\push'\meas'}$. Since $\iota$ is distance-preserving, $\smash{\tilde{\pi}}$ obeys  \eqref{Eq:Gleichheit} for the respective quotient time separation functions $\smash{\tilde{\tau}}$ and $\smash{\tilde{\tau}'}$. To construct a suitable coupling $\pi$ of $\meas$ and $\meas'$, we will use the disintegration \cref{Th:DisintTheorem}. It yields the existence of Borel measurable maps $\smash{\mu_\cdot \colon \tilde{\mms} \to \Prob(\mms)}$ and $\smash{\mu_\cdot'\colon \tilde{\mms}'\to\Prob(\mms')}$ with
\begin{align*}
\rmd \meas(x) &= \int_{\tilde{\mms}} \rmd \mu_{\tilde{x}}(x) \d \tilde{\meas}(\tilde{x}),\\
\rmd \meas'(x') &= \int_{\tilde{\mms}'} \rmd \mu_{\tilde{x}'}'(x') \d\tilde{\meas}'(\tilde{x}').
\end{align*}
In turn, we define $\pi\in\Prob(\mms\times\mms')$ by
\begin{align*}
\rmd\pi(x,x') := \int_{\tilde{\mms}\times\tilde{\mms}'}\rmd\mu_{\tilde{x}}(x)\,\rmd\mu_{\tilde{x}'}(x')\d\tilde{\pi}(\tilde{x},\tilde{x}').
\end{align*}
It is readily verified $\pi$ is a coupling of $\meas$ and $\meas'$. Furthermore, by construction of the distance quotient and $\smash{\tilde{\pi}}$, $\smash{\pi^{\otimes 2}}$-a.e.~$(x,x',y,y')\in (\mms\times\mms')^2$ satisfies \eqref{Eq:Gleichheit} since the conditional laws $\smash{\mu_{\tilde{x}}}$ and $\smash{\mu_{\tilde{x}'}}$ are concentrated on $\smash{q^{-1}(\tilde{x})}$ and $\smash{q'^{-1}(\tilde{x}')}$ for $\smash{\tilde{\meas}}$-a.e.~$\smash{\tilde{x}\in\supp\tilde{\meas}}$ and $\smash{\tilde{\meas}'}$-a.e.~$\smash{\tilde{x}'\in\supp\tilde{\meas}'}$, respectively.

\ref{La:222} $\Longrightarrow$ \ref{La:333}. Let $\pi$ be a coupling of $\meas$ and $\meas'$ such that \eqref{Eq:Gleichheit} is satisfied. Then, define $\smash{\hat\mms} := \supp (q,q')_\push\pi \cup \{(i^0,i'^0)\}\subset \tilde{\mms}\times\tilde{\mms}'$, $\smash{\hat\meas := (q,q')_\push\pi}$, as well as the function $\smash{\hat\tsep\colon \hat\mms^2\to \R_+}$ by
\begin{align*}
2\,\hat\tsep((x,x'),(y,y')) := \tsep(x,y) +\tsep'(x',y').
\end{align*}
Evidently, $\smash{\hat\mms}$ is normalized, has a spacelike boundary $\smash{\hat{\imath}^0 := (i^0,i'^0)}$, and satisfies $\smash{\supp\hat\meas \cup \{\hat i^0\} = \hat\mms}$. It suffices to check the three axioms of a bounded Lorentzian metric space from \cref{Def:Bounded Lorentzian metric spaces} for the tuple $\smash{(\hat\mms,\hat\tsep)}$; the  maps $\hat\iota$ and $\hat\iota'$ are then simply chosen as the respective projections.
\begin{enumerate}[label=\textnormal{\alph*.}]
\item To show the reverse triangle inequality, assume $(x,x'),(y,y'),(z,z')\in \hat{\mms}$ are points satisfying  \eqref{Eq:Gleichheit} with $\smash{\hat{\tsep}((x,x'),(y,y')) > 0}$ and $\smash{\hat{\tsep}((y,y'),(z,z')) > 0}$. Thanks to  \eqref{Eq:Gleichheit}, we have
\begin{align*}
0<\hat\tsep((x,x'),(y,y')) &= \tsep(x,y) = \tsep'(x',y'),\\
0<\hat\tsep((y,y'),(z,z')) &= \tsep(y,z) = \tsep'(y',z').
\end{align*}
By continuity of the involved time separation functions and since sets of conegligible $\smash{\pi^{\otimes 2}}$-measure are dense in $\supp\pi$, the above holds for all points $(x,x'),(y,y'),(z,z')\in \smash{\hat\mms}$ with the properties that  $\smash{\hat{\tsep}((x,x'),(y,y'))>0}$ and $\smash{\hat{\tsep}((y,y'),(z,z'))>0}$. Consequently, we obtain $x\ll y \ll z$ and $x'\ll' y'\ll'z'$ and therefore
\begin{align*}
2\,\hat\tsep((x,x'),(z,z')) &= \tsep(x,z) + \tsep(x',z')\\
&\geq \tsep(x,y) + \tsep(y,z) + \tsep(x',y') + \tsep(y',z')\\
&= 2\,\hat\tsep((x,x'),(y,y')) + 2\,\hat\tsep((y,y'),(z,z')).
\end{align*}
\item Using the relative product topology on $\smash{\hat{\mms}}$ gives the sought topology for which $\smash{\hat\tsep}$ is continuous. As concerns compactness of superlevel sets, the set $\{\hat\tsep\geq \varepsilon\}$ is the intersection of the two sets $\{\tsep\circ\pr_{13}\geq \varepsilon\}$ and $\{\tsep'\circ\pr_{24}\geq \varepsilon\}$ for every $\varepsilon >0$, and hence itself compact (where we used again that \eqref{Eq:Gleichheit} holds everywhere).
\item Lastly, note  separability of $\mms$ and $\mms'$ transfers to $\smash{\hat\mms}$. Thus, let $\smash{\hat{D}\subset\hat\mms}$ be countable and dense. Then the projections  $\smash{D := \pr_1(\hat{D}) \subset\supp q_\push\meas \cup\{i^0\}}$ and $\smash{D' := \pr_2(\hat{D})\subset\supp q'_\push\meas'\cup\{i'^0\}}$ are countable and dense. Now assume  $\smash{(x,x') \neq (y,y')}$ for two given points $\smash{(x,x'),(y,y')\in \hat\mms}$ and thus, without loss of generality, that $x\neq y$. Then by Minguzzi--Suhr  \cite[Prop.~1.10]{MS:22b} there exists $z\in D$ such that $\tsep(x,z) \neq \tsep(y,z)$, again without loss of generality. Choosing $z'\in S'$ with $\smash{(z,z')\in \hat D}$, it remains to remark that
\begin{align*}
\hat\tsep((x,x'),(z,z')) &= \tsep(x,z) \neq \tsep(y,z) = \hat\tsep((y,y'),(z,z')).
\end{align*}
\end{enumerate}

\ref{La:333} $\Longrightarrow$ \ref{La:111}. Taking distance quotients and considering the quotients induced by $\smash{\hat{\iota}}$ and $\smash{\hat{\iota}'}$ easily yields $\scrM$ and $\scrM'$ are isomorphic to $\smash{\hat{\scrM}}$. Since isomorphy is an equivalence relation, we are done.
\end{proof}

\subsection{Parametrization}\label{Sub:Parametrization} A further useful consequence of the considered topologies being Polish is that bounded Lorentzian metric measure spaces are \emph{standard Borel spaces}, cf.~e.g.~Srivastava \cite[Ch.~3]{Sri:98} for background and details. In particular, the concept of \emph{parametri\-zation} applies to normalized bounded Lorentzian metric measure spaces $\scrM$ according to \cref{Le:Parametrization} below. Already in the general setting, this leads to nice representation results; however, since only the reference measure is preserved, often geometric interpretation gets lost.

\begin{definition}[Parametrization] A \emph{parametrization} of a normalized measure space $(\mms,\meas)$ is a Borel measurable map $\psi\colon[0,1]\to\mms$ such that
\begin{align*}
\meas = \psi_\push\Leb^1.
\end{align*}
\end{definition}

The first statement of the following result is due to Srivastava \cite[Thm.~3.4.23]{Sri:98}, the second one can be found in Sturm \cite[Lem.~1.15]{Stu:12}.

\begin{lemma}[Parametrizations of bounded Lorentzian metric measure spaces]\label{Le:Parametrization} Let $\scrM$ and $\scrM'$ be two normalized bounded Lorentzian metric measure spaces. Then $\scrM$ admits a parametrization $\psi\colon[0,1]\to\mms$, symbolically $\psi\in\Par(\scrM)$.

In addition, a Borel probability measure $\pi\in\Prob(\mms\times\mms')$ is a coupling of $\meas$ and $\meas'$ if and only if there exist $\psi\in \Par\scrM$ and $\psi'\in\Par\scrM'$ such that
 \begin{align*}
 \pi = (\psi,\psi')_\push\scrL^1.
 \end{align*}
 \end{lemma}

 It is clear that for an isomorphism $\smash{\iota\colon \supp q_\push\meas\cup\{i^0\} \to \supp q'_\push\meas'\cup\{i'^0\}}$ we have that $\iota\circ\psi \in \Par\scrM'$ if and only if $\psi\in\Par\scrM$.

Given $\psi\in\Par\scrM$ let $\smash{\psi^*\tsep  \colon I^2 \to [0,\infty)}$ be the \emph{pullback} of $\tsep$ under $\psi$, i.e.,
\begin{align}\label{Eq:Pullback}
\psi^*\tsep(s,t) := \tsep(\psi(s),\psi(t)).
\end{align}
Note that $\psi^*\tsep$ does not induce a bounded Lorentzian metric structure on $[0,1]$ in general, as it is only measure-preserving.

The following then transfers from Sturm \cite[Rem.~1.16]{Stu:12} to our setting.

 \begin{remark}[Bijectivity] If $\scrM$ has no atoms, i.e.~the reference measure $\meas$ does not give mass to singletons, there exists $\psi\in\Par\scrM$ which is bijective and has a Borel measurable inverse. In the presence of atoms, there is $\psi\in\Par\scrM$ which restricts to a bijective map with Borel measurable inverse ``after $\meas$ has made all its jumps by charging singletons''; see Sturm \cite[Rem.~1.16]{Stu:12} for details.
 \end{remark}

\section{Gromov's reconstruction theorem}\label{Sub:GromovRecon}

This chapter is devoted to the proof of the Lorentzian Gromov reconstruction \cref{Th:Gromov reconstruction}. We concentrate on normalized spaces. Since isomorphic bounded Lorentzian metric measure spaces have the same mass, this class suffices  to be considered to cover the general case of merely finite measures.

Recall the \cref{Def:Poly} of polynomials.

\begin{theorem}[Lorentzian Gromov reconstruction theorem]\label{Th:Gromov reconstruction finite measure} Let $\scrM$ and $\scrM'$ be normalized bounded Lorentzian metric spaces. Suppose $\Phi(\scrM) = \Phi(\scrM')$ for every polynomial $\Phi$. Then $\scrM$ and $\scrM'$ are isomorphic.
\end{theorem}

\begin{remark}[Invariance of polynomials under isomorphy] In fact, the converse of \cref{Th:Gromov reconstruction finite measure} is also true: if $\scrM$ and $\scrM'$ are isomorphic, then $\Phi(\scrM) = \Phi(\scrM')$ for every polynomial $\Phi$. Indeed, let $k\in\N$ and $\smash{\varphi\in \Cont_\bounded(\R^{k\times k})}$ be given. Moreover, let $\pi$ be a coupling of $\meas$ and $\meas'$ according to \cref{Le:Char iso}. Then
\begin{align*}
&\int_{\mms^k} \varphi \circ\sfT^k \d\meas^{\otimes k} - \int_{\mms'^k}\varphi\circ\sfT'^k\d\meas'^{\otimes k}\\
&\qquad\qquad = \int_{(\mms\times\mms')^k} \d\pi^{\otimes k}(x_1,x_1',\dots,x_k,x_k')\\
&\qquad\qquad\qquad\qquad \big[\varphi\circ\sfT^k(x_1,\dots,x_k) - \varphi\circ \sfT'^k(x_1',\dots,x_k')\big]\\
&\qquad\qquad =0
\end{align*}
thanks to the relation \eqref{Eq:Gleichheit}.
\end{remark}

To prove \cref{Th:Gromov reconstruction finite measure}, we follow the proof variant for Gromov's classical   reconstruction theorem for metric measure spaces \cite[§3$\smash{\frac{1}{2}}$.7]{Gro:99} (divided into four steps)  which uses the beautiful probabilistic argument due to Vershik \cite{Ver:03}  worked out in more detail by Kondo \cite{Kon:05}, to where we refer for the technicalities.

We first introduce some notation. Let $\G^\infty$ be the projective limit of $\smash{(\G^k)_{k\in\N}}$ with the induced topology. Furthermore, let $\smash{\mms^{\N}}$ denote the space of all $\mms$-valued sequences $(x_i)_{i\in\N}$, and let $\smash{\meas^{\otimes \infty}}$ be the product measure of $\meas$ on $\smash{\mms^{\N}}$. (In probabilistic words, $\smash{\meas^{\otimes \infty}}$ is the joint distribution of a sequence $(X_i)_{i\in\N}$ of i.i.d.~$\mms$-valued random variables with law $\meas$.) Define $\sfT^\infty\colon \mms^{\N} \to \G^\infty$ componentwise by 
\begin{align*}
\sfT^\infty((x_l)_{l\in\N})_{i,j} := \tsep(x_i,x_j).
\end{align*}
Note, $\sfT^\infty$ is continuous, hence Borel measurable. Finally, we consider
\begin{align*}
\bar{\meas}^\infty := \sfT^\infty_\push\meas^{\otimes\infty},
\end{align*}
which is a probability measure on $\G^\infty$. Analogously, we define $\mms'^{\N}$ and $\smash{\bar{\meas}'^\infty}$.

\begin{lemma}[Step A]\label{Le:m1 = m2} The assumptions of \cref{Th:Gromov reconstruction finite measure} imply 
\begin{align*}
\bar{\meas}^\infty = \bar{\meas}'^\infty.
\end{align*}
\end{lemma}

\begin{proof} The short argument is that $\smash{\bar{\meas}^\infty}$ and $\smash{\bar{\meas}'^\infty}$ coincide with the unique projective limits of the sequences $\smash{(\bar{\meas}^k)_{k\in\N}}$ and $\smash{(\bar{\meas}'^k)_{k\in\N}}$ obtained by Kolmogorov's extension theorem; since $\smash{\bar{\meas}^k = \bar{\meas}'^k}$ for every $k\in\N$ by our hypothesis, the claim follows. We refer to Kondo  \cite[Lem.~2.2]{Kon:05} for the technical details in the metric measure case which carry over to our setting with no change.
\end{proof}

Let $\smash{E \subset\mms^{\N}}$ be the set of all \emph{generic sequences} as defined by Gromov \cite[§3$\smash{\frac{1}{2}}$.22]{Gro:99}, i.e.~all sequences $(x_i)_{i\in\N}\in \mms^{\N}$ such that for every $f\in\Cont_\bounded(\mms)$,
\begin{align}\label{Eq:gener}
\lim_{n\to\infty} \frac{1}{n}\sum_{i=1}^n f(x_i) = \int_\mms f\d\meas.
\end{align}
The property of a sequence being generic remains unaltered by adding or removing finitely many points to of from it (e.g.~a spacelike boundary, some puncture, etc.).

\begin{remark}[Density of generic sequences]\label{Re:dense} Every generic sequence $(x_i)_{i\in\N}$ with $x_i\in \supp\meas$ for every $i\in\N$ is dense in $\supp\meas$. Let  $\met$ be a  Polish metric inducing the topology of $\mms$. Given any $x\in\supp\meas$, let $B_r(x)\subset\mms$ be the open ball of radius $r>0$ centered at $x$ with respect to $\met$, and consider $\smash{f := (r - \met(x,\cdot))^+\in\Cont_\bounded(\mms)}$. Then $f>0$ on $B_r(x)$ and $\meas[B_r(x)]>0$ easily imply $\smash{\int_\mms f\d\meas > 0}$. This together with \eqref{Eq:gener} forces $f(x_i) > 0$ for infinitely many $i\in\N$, hence $x_i\in B_r(x)$.
\end{remark}

\begin{lemma}[Step B]\label{Le:mE} The set $E$ is $\smash{\scrB^{\otimes\infty}}$-measurable and 
\begin{align*}
\meas^{\otimes\infty}[E]=1.    
\end{align*}
\end{lemma}

\begin{proof} We sketch the key and elegant argument for the claimed identity. Define the probability space $(\Omega,\scrF,\mathbf{P})$, where $\smash{\Omega := \mms^{\N}}$, $\scrF := \scrB^{\otimes\infty}$, and $\smash{\mathbf{P} := \meas^{\otimes \infty}}$. Given any $f\in \Cont_\bounded(\mms)$, the sequence $(X_i)_{i\in\N}$ of bounded random variables $X_i \colon \Omega\to \R$ with $\smash{X_i := f\circ\pr_i}$ is i.i.d.~and has constant expectation $\smash{\mathbf{E}[X_i] = \int_{\mms} f\d\meas}$. The strong law of large numbers applied to $(X_i)_{i\in\N}$  implies 
\begin{align*}
\meas^{\otimes\infty}[E_f] = 1   
\end{align*}
subject to the $f$-dependent, $\scrB^{\otimes \infty}$-measurable set
\begin{align*}
E_f := \Big\lbrace (x_i)_{i\in\N} \in \mms^{\N} : \lim_{n\to\infty} \frac{1}{n}\sum_{i=1}^n f(x_i) = \int_\mms f\d\meas \Big\rbrace.
\end{align*}
Using $\mms$ is separable and metrizable, one now shows $E$ is effectively the union of $E_f$ over \emph{countably many} $f$ as above. This easily yields the claimed identity. We refer to  Kondo \cite[Lem.~2.4, Lem.~2.5]{Kon:05} for details in the metric measure case as well as the $\smash{\scrB^{\otimes \infty}}$-measurability of $E$.
\end{proof}

Analogous constructions and observations apply relative to $\mms'$.

Since the intersection of two sets with probability one still has probability one\footnote{For this argument, it would actually suffice to know $\smash{\meas^{\otimes \infty}[E] > 1/2}$ and $\smash{\meas'^{\otimes \infty}[E'] >1/2}$.}, \cref{Le:m1 = m2} and \cref{Le:mE} --- together with the definitions of $\smash{\bar{\meas}^\infty}$ and $\smash{\bar{\meas}'^\infty}$ --- readily imply the following.

\begin{corollary}[Step C]\label{Cor:generic} The assumptions of \cref{Th:Gromov reconstruction finite measure} imply the existence of generic sequences $\smash{(x_i)_{i\in\N}\in \mms^{\N}}$ and $\smash{(x_i')_{i\in\N} \in \mms'^{\N}}$ such that for every $i,j\in\N$ we have $x_i\in\supp\meas$,  $x_i'\in\supp\meas'$, and
\begin{align*}
\tsep(x_i,x_j) = \tsep'(x_i',x_j').
\end{align*}
\end{corollary}

\begin{proof}[Proof of \cref{Th:Gromov reconstruction finite measure} \textnormal{(Step D)}] Let $\smash{(x_i)_{i\in\N}\in\mms^{\N}}$ and $\smash{(x'_i)_{i\in\N}\in\mms'^{\N}}$ be generic sequences as given by \cref{Cor:generic}. By \cref{Re:dense}, the sets $D:=\{x_i:i\in\N\}$ and $D':=\{x_i':i\in\N\}$ are dense in $\supp\meas$ and $\supp\meas'$, respectively. Consider the quotient maps $\smash{q\colon\supp\meas\to \supp\tilde{\meas}}$ and $\smash{q'\colon\supp\meas'\to\supp\tilde{\meas}'}$ from \cref{Re:Distquot}, where $\smash{\tilde{\meas} := q_\push\meas}$ and $\smash{\tilde{\meas}' := q_\push'\meas'}$.  Note that if distinct $i,j\in\N$ obey  $q(x_i) = q(x_j)$ then $\smash{q'(x_i') = q'(x_j')}$. Indeed, suppose to the contrary  $\smash{q'(x_i') \neq q'(x_j')}$. Since $\smash{\supp\tilde{\meas}'}$ is a bounded Lorentzian metric space by construction and $q(D')$ is dense in it (as the  image of a dense set under a continuous and surjective map), using Minguzzi--Suhr \cite{MS:22b}*{Prop.~1.11} there is  $l\in\N$ such that without restriction (up to swapping the entries) we have $\smash{\tilde{\tau}(q'(x_i'),q'(x_l')) \neq \tilde{\tau}'(q'(x_j'),q'(x_l'))}$. By definition of the quotient time separation functions and \cref{Cor:generic}, this would yield 
\begin{align*}
0 &= \tilde{\tau}(q(x_i),q(x_l)) - \tilde{\tau}(q(x_j),q(x_l))\\
&= \tau(x_i,x_l) - \tau(x_j,x_l) \\
&= \tau'(x_i',x_l') -\tau'(x_j',x_l')\\
&= \tilde{\tau}'(q'(x_i'),q'(x_l')) - \tilde{\tau}'(q'(x_j'),q'(x_l'))\\
&\neq 0,
\end{align*}
which is a contradiction. Analogously, we prove that if distinct $i,j\in\N$ satisfy $\smash{q'(x_i') = q'(x_j')}$ then $\smash{q(x_i) = q(x_j)}$.

Then let us define $\jmath\colon D\to D'$ by $\smash{\jmath(x_i) := x_i'}$ for every $i\in\N$. By construction, $\jmath$ is distance-preserving.  Moreover, by the above discussion its natural quotient map $\smash{\iota\colon q(D) \to q(D')}$ is well-defined and, by construction, distance-preserving. Also, given any $i\in\N$ it clearly satisfies $\smash{\iota\circ q(x_i) = q'(x_i')}$.  As the domain and codomain of $\iota$ are dense in bounded Lorentzian metric spaces, \cref{Le:Extension} implies that $\iota$ extends to a nonrelabeled isometry $\smash{\iota\colon\supp\tilde{\meas}\cup\{i^0\} \to \supp\tilde{\meas}'\cup\{i'^0\}}$; in particular, $\iota$ is continuous. It remains to prove $\iota$ preserves the quotient reference measures. But given any $\smash{g\in \Cont_\bounded(\tilde{\mms}')}$, noting $\smash{g\circ q'\in \Cont_\bounded(\mms')}$ and $\smash{g\circ\iota\circ q\in \Cont_\bounded(\mms)}$ we get
\begin{align*}
\int_{\tilde{\mms}'} g \d\tilde{\meas}' &= \int_\mms g\circ q'\d\meas'\\
&= \lim_{n\to\infty} \frac1n\sum_{i=1}^n g\circ q'(x_i')\\
&= \lim_{n\to\infty} \frac1n\sum_{i=1}^n g\circ \iota\circ q(x_i)\\
&= \int_{\tilde{\mms}} g\circ\iota \d\tilde{\meas}.
\end{align*}
This concludes the proof.
\end{proof}

\section{Measured Lorentz--Gromov--Hausdorff convergence}\label{Sub:mLGH}
This part introduces three notions of convergence of the space $\MM_1$ of all isomorphism classes of normalized bounded Lorentzian metric measure spaces.

Without further notice, in the sequel we fix a sequence $[\scrM_n]_{n\in\N}$ in $\MM_1$ as well as elements $[\scrM_\infty], [\scrM], [\scrM'], [\scrM''] \in\MM_1$. Typically, the isomorphism class of $\scrM_\infty$  plays the role of limit of $[\scrM_n]_{n\in\N}$ in a sense to be made precise; the other three structures mainly appear in statements about  adjacent basic properties such as the triangle inequality or the invariance under isomorphisms.

\subsection{Intrinsic approach by the reconstruction \cref{Th:Gromov reconstruction}}\label{Sub:Intrinsic approach} Our first proposal consists of the following  \cref{Def:Intrinsic,Def:LGweaktop} which mirror the so-called Gromov weak topology of isomorphism classes of normalized metric measure spaces by Greven--Pfaffelhuber--Winter \cite{GPW:09}*{Def.~2.8}.

\begin{definition}[Intrinsic convergence]\label{Def:Intrinsic} We say $\smash{[\scrM_n]_{n\in\N}}$ \emph{converges   intrinsically} to $[\scrM_\infty]$ if for every $k\in\N$, the sequence $\smash{(\bar{\meas}_n^k)_{n\in\N}}$ of matrix laws on $\smash{\R^{k\times k}}$ converges narrowly to $\smash{\bar{\meas}^k_\infty}$, where $\smash{\bar{\meas}_n^k}$ is from \eqref{Eq:barm^k} for every $n\in\N\cup\{\infty\}$.
\end{definition}

By our reconstruction \cref{Th:Gromov reconstruction},  intrinsic convergence of normalized bounded Lorentzian metric measure spaces only depends on their isomorphism classes.

\begin{definition}[Lorentz--Gromov weak topology]\label{Def:LGweaktop} The  topology on $\MM_1$ induced by intrinsic convergence is called \emph{Lorentz--Gromov weak topology}.
\end{definition}

The rest of \cref{Sub:Intrinsic approach} establishes basic properties of this topology. 

\begin{remark}[Continuity of polynomials] Every polynomial is  continuous in the Lorentz--Gromov weak topology. This  is a direct consequence of \cref{Def:Intrinsic}.
\end{remark}

\begin{remark}[Metrizability]\label{Re:Metrizability} The Lorentz--Gromov weak topology  is induced by a metric. Indeed, given any $k\in\N$,   the narrow topology of $\smash{\Prob(\R^{k\times k})}$ is metrizable. We recall an explicit construction of a metric from Ambrosio--Gigli--Savaré  \cite{AGS:08}*{Rem. 5.1.1}: by taking e.g.~the customary $L^\infty$-norm on $\smash{\R^{k\times k}}$, they construct a countable set $\smash{\scrC_k := \{\varphi_i^k : i\in\N\} \subset\Cont_\bounded(\R^{k\times k})}$  such that for every $i\in\N$, we have $\smash{\vert \varphi_i^k\vert\leq 1}$  on $\R^{k\times k}$ and $\Lip\,\varphi_i^k\leq 1$. Then 
\begin{align*}
\met_k(\mu,\nu) := \sum_{i\in\N} 2^{-i}\,\Big\vert\!\int_{\R^{k\times k}} \varphi_i^k\d\mu - \int_{\R^{k\times k}}\varphi_i^k\d\nu\Big\vert
\end{align*}
defines a complete and separable metric inducing the narrow topology of $\Prob(\R^{k\times k})$.

Then $[\scrM_n]_{n\in\N}$ converges intrinsically to $[\scrM_\infty]$ if and only if $\sfD(\scrM_n,\scrM_\infty)\to 0$ as $n\to\infty$, where the metric $\sfD$ on $\MM_1$ is defined by
\begin{align*}
\sfD(\scrM,\scrM') :=\sum_{k\in \N} 2^{-k}\,\met_k(\bar{\meas}^k,\bar{\meas}'^k).
\end{align*}
Again by our reconstruction \cref{Th:Gromov reconstruction}, the quantity $\sfD(\scrM,\scrM')$ depends only on the  isomorphism classes of $\scrM$ and $\scrM'$, respectively. 
\end{remark}

\begin{example}[Intrinsic convergence of Chru\'sciel--Grant approximations of continuous spacetimes] A compact, causally convex subset without spacelike boundary of a globally hyperbolic spacetime with a continuous metric is a bounded Lorentzian metric space, see Minguzzi--Suhr \cite[Thm.\ 2.5]{MS:22b}, since the time separation function is continuous, cf.~Sämann \cite{Sae:16}, Kunzinger--Sämann \cite{KS:18}, and Ling \cite{Lin:24}. Restricting and normalizing the volume measure to such a subset yields a normalized bounded Lorentzian metric measure space. Now,  Chru\'sciel--Grant \cite{CG:12} showed that one can approximate a continuous metric $g$ with sequences of smooth Lorentzian metrics $(\check{g}_n)_{n\in\N}$ and $(\hat{g}_n)_{n\in\N}$ that converge locally uniformly to $g$ and have nested lightcones, i.e.~$\check{g}_n\preceq g \preceq\hat{g}_n$ for all $n\in\N$. Here, for two Lorentzian metrics $g',g''$, we write $g'\preceq g''$ if $g''(v,v)\leq 0$ for all $v\in TM$ with $g'(v,v)\leq 0$. One can cover a globally hyperbolic spacetime by compact and causally convex sets, see Mondino--Sämann \cite[Lem.~3.11]{MS:25} and one can use a refined approximation from the outside, which is given in McCann--Sämann \cite[§A]{McCS:22}.
Then, for each $n\in \N$ we have a normalized bounded Lorentzian metric space $\scrM_n$ such that the time separation functions converge uniformly (for a similar argument see Mondino--Sämann  \cite[Thm.\ 10.1]{MS:25}). Since the metrics converge uniformly we get that the induced volume measures also converge. It is thus straightforward to show that for every polynomial the sequence of matrix laws $\smash{\bar{\meas}_n^k}$  converges narrowly to $\smash{\bar{\meas}^k_\infty}$ as $n\to\infty$ for every $k\in\N$, where $\smash{\bar{\meas}_n^k}$ is from \eqref{Eq:barm^k} for  $n\in\N\cup\{\infty\}$, hence $\scrM_n\to \scrM_\infty$ intrinsically as $n\to\infty$.
\end{example}

\begin{proposition}[Density of causets]\label{Pr:Finite approx} The totality of all isomorphism classes of normalized bounded Lorentzian metric measure spaces with finite cardinality (i.e., causets) is dense in $\MM_1$ with respect to intrinsic convergence.
\end{proposition}

\begin{proof} Let $[\scrM_\infty]\in\MM_1$ be given. By \cref{Re:Goodrepr}, without loss of generality we may and will choose the representative $\scrM$ to have full support. As is well-known, $\meas_\infty$ can be approximated  in the narrow topology by a sequence $(\mu_n)_{n\in\N}$ in $\Prob(\mms_\infty)$ whose supports are finite (cf.\ e.g.\ \cite[Thm.\ 15.10]{AB:06}). Given any $n\in\N$, let $\smash{\mms_n}$ be the distance quotient of $\supp\mu_n$ from \cref{Re:Distquot} with quotient map $q_n\colon\supp\mu_n \to \mms_n$. Let $\tau_n$ be the  quotient time separation function induced from the restriction of $\tau$ to $\smash{(\supp\mu_n)^2}$ and set $\smash{\meas_n := (q_n)_\push\mu_n}$. This yields a finite bounded Lorentzian metric measure space $\scrM_n$.

We claim that the sequence $[\scrM_n]_{n\in\N}$ converges intrinsically to $[\scrM_\infty]$. Indeed, for every $k\in\N$ the sequence $\smash{(\meas_n^{\otimes k})_{n\in\N}}$ converges narrowly to $\smash{\meas^{\otimes k}}$ in $\smash{\Prob(\mms^k)}$. Let $\smash{q_n^k\colon (\supp\meas_n)^k \to \mms_n^k}$ designate the $k$-fold product of $q_n$. Given any $\smash{\varphi\in\Cont_\bounded(\R^{k\times k})}$, we have $\smash{\varphi\circ\sfT_\infty^k\in \Cont_\bounded(\mms_\infty^k)}$ and 
\begin{align*}
\lim_{n\to\infty} \int_{\mms_n^k} \varphi\circ \sfT_n^k\d\meas_n^{\otimes k} &= \lim_{n\to\infty} \int_{(\supp\mu_n)^k} \varphi\circ \sfT_n^k \circ q_n^k\d\mu_n^{\otimes k}\\
&= \lim_{n\to\infty} \int_{\mms_\infty^k} \varphi\circ \sfT_\infty^k \d\mu_n^{\otimes k}\\
&= \int_{\mms_\infty^k} \varphi\circ\sfT_\infty^k\d\meas_\infty^{\otimes k}.
\end{align*}
This finishes the proof.
\end{proof}

\begin{corollary}[Separability]\label{Cor:LGw separable} The Lorentz--Gromov weak topology is separable.
\end{corollary}

However, the distance from \cref{Re:Metrizability} is not complete. The counterexample to follow is analogous to the one by  Greven--Pfaffelhuber--Winter \cite{GPW:09}*{Ex.~2.12} for metric measure spaces.

\begin{example}[Incompleteness]\label{Ex:no limit} Given any  $n\in\N$, we define $\mms_n := \{1,\dots,2^n\}$. In addition,  define $\smash{\tsep_n \colon \mms_n^2 \to \R_+}$ by
\begin{align*}
\tsep_n(x,y) := \begin{cases} 1 + 4^{-n+1}\,(y-2^{n-1}) & \textnormal{if }x=1\neq y,\\
0 & \textnormal{otherwise}. 
\end{cases}
\end{align*}
Then $(\mms_n,\tsep_n)$ is a bounded Lorentzian metric space --- the different values of $\tsep_n$ outside $0$ yield the point distinction property --- with the discrete topology. Let $\meas_n$ be the uniform distribution on $\mms_n$. Then  $(\scrM_n)_{n\in\N}$ is a $\sfD$-Cauchy sequence. However, $(\bar{\meas}_n^k)_{n\in\N}$ would converge to the $k$-fold product of a uniform distribution on an infinite discrete set, which does not exist.
\end{example}

\begin{remark}[Polynomials are not dense] Using  \cref{Ex:no limit}, as in Löhr \cite[Rem. 2.6]{Loe:13} one  shows  the space of polynomials is not uniformly dense in $\Cont_\bounded(\MM_1)$.
\end{remark}

\subsection{Distortion distance}\label{Sub:Lorentz Gromov Prokhorov} We now continue with a distance that measures the distortion of two bounded Lorentzian metric spaces. Among our proposals, this mirrors Müller's \cite{Mue:22} and Minguzzi--Suhr's \cite{MS:22b} approach to Lorentz--Gromov--Hausdorff convergence \emph{without} reference measures most closely; cf.~\cref{Sub:RelGromHaus} below for a discussion about precise relations. Our approach reflects  previous works from metric measure geometry. Notably, we adapt Greven--Pfaffelhuber--Winter's \emph{Eurandom distance} \cite{GPW:09}*{§10} (which was called \emph{$L^0$-distortion distance} in Sturm \cite{Stu:12}). It is somewhat similar to the Fan metric  \eqref{Eq:Fan}. Variants of this ansatz to convergence in $\MM_1$  are discussed in \cref{Sub:Strong} below.

\begin{definition}[Distortion distance]\label{Def:L0dist} We define the \emph{distortion distance} of $\scrM$ and $\scrM'$ through
\begin{align*}
\LL\bdDelta_0(\scrM,\scrM') := \inf \inf\!\big\lbrace\varepsilon > 0 : \pi^{\otimes 2}\big[\big\lbrace (x,x',y,y') \in (\mms\times\mms')^2 : \qquad\qquad& \\
 \big\vert\tau(x,y) - \tau'(x',y')\big\vert > \varepsilon\big\rbrace\big] \leq \varepsilon\big\}, &
\end{align*}
where the outer infimum is taken over all couplings $\pi$ of $\meas$ and $\meas'$.
\end{definition}

Note this distance takes values in $[0,1]$.

To proceed, we show the existence of optimal couplings in the above definition. Then we verify the distortion distance constitutes a metric  on $\MM_1$. This induces a natural notion of convergence on $\MM_1$, cf.~\cref{Def:Distortion convergence} below. 

\begin{theorem}[Existence of optimal couplings for $\LL\bdDelta_0$]\label{Th:ExOptL0q} There exists a coupling $\pi$ of $\meas$ and $\meas'$ which is optimal in   the above definition of $\smash{\LL\bdDelta_0(\scrM,\scrM')}$, i.e.
\begin{align*}
\LL\bdDelta_0(\scrM,\scrM') = \inf\!\big\lbrace \varepsilon > 0 : \pi^{\otimes 2}\big[\big\{(x,x',y,y')\in (\mms\times\mms')^2 :\qquad\qquad&\\ 
 \big\vert\tau(x,y) - \tau'(x',y')\big\vert > \varepsilon\big\}\big] \leq \varepsilon\big\rbrace.&
\end{align*}

Moreover, the latter infimum is attained, i.e.\footnote{By monotone convergence, it is easy to see \emph{every} optimal coupling satisfies  this property.}
\begin{align}\label{Eq:Inf!}
\begin{split}
&\pi^{\otimes 2}\big[\big\{(x,x',y,y') \in (\mms\times\mms')^2 :\\
&\qquad\qquad \big\vert\tau(x,y)-\tau'(x',y')\big\vert > \LL\bdDelta_0(\scrM,\scrM')\big\}\big] \leq \LL\bdDelta_0(\scrM,\scrM').
\end{split}
\end{align}
\end{theorem}

\begin{proof} We will show both claims simultaneously. Throughout the proof, for $\varepsilon > 0$ we consider the open set
\begin{align*}
U_\varepsilon := \big\lbrace (x,x',y,y') \in (\mms\times\mms')^2 :\big\vert\tau(x,y)-\tau'(x',y')\big\vert > \varepsilon\big\rbrace.
\end{align*}

Let $(\pi_n)_{n\in\N}$ be a sequence of couplings of $\meas$ and $\meas'$ such that
\begin{align}\label{Eq:Optsequ}
\LL\bdDelta_0(\scrM,\scrM') = \lim_{n\to\infty} \inf\!\big\{\varepsilon >0 : \pi_n^{\otimes 2}[U_\varepsilon] \leq \varepsilon \big\}.
\end{align}
Since the marginal sequences are  tight (as singletons), by Prokhorov's theorem $(\pi_n)_{n\in\N}$ has a nonrelabeled subsequence which converges narrowly to a coupling $\pi$ of $\meas$ and $\meas'$. If $\smash{\LL\bdDelta_0(\scrM,\scrM') = 1}$ we are done. Otherwise, let $\varepsilon\in (0,1)$ satisfy  $\varepsilon_0:=\smash{\LL\bdDelta_0(\scrM,\scrM') < \varepsilon}$. Recall the sequence $\smash{(\pi_n^{\otimes 2})_{n\in\N}}$ of product measures converges narrowly to $\smash{\pi^{\otimes 2}}$. Combining  lower semicontinuity of open sets with respect to  the narrow topology  with \eqref{Eq:Optsequ} therefore yields 
\begin{align*}
\pi^{\otimes 2}[U_\varepsilon] \leq \liminf_{n\to\infty}\pi_n^{\otimes 2}[U_\varepsilon] \leq \varepsilon.
\end{align*}
Sending $\smash{\varepsilon \to \varepsilon_0+}$ and using monotone convergence,
\begin{align}\label{Eq:Huh}
&\pi^{\otimes 2}[U_{\varepsilon_0}] \leq \varepsilon_0.
\end{align}
In turn, this yields
\begin{align*}
\LL\bdDelta_0(\scrM,\scrM')  &\leq \inf\!\big\lbrace \varepsilon > 0 : \pi^{\otimes 2}[U_\varepsilon] \leq \varepsilon\big\rbrace \leq \varepsilon_0,
\end{align*}
where in the first bound we used $\pi$ as a competitor in the definition of $\LL\bdDelta_0(\scrM,\scrM')$ and \eqref{Eq:Huh} in the second inequality. The desired statements  follow.
\end{proof}

We anticipatively define $\LL\bdDelta_0$ on $\MM_1$ by
\begin{align}\label{Eq:Llbox0}
\LL\bdDelta_0([\scrM],[\scrM']) := \LL\bdDelta_0(\scrM,\scrM').
\end{align}

\begin{theorem}[Metric property of $\LL\bdDelta_0$]\label{Th:Metric prop} The quantity \eqref{Eq:Llbox0} constitutes a well-defined metric on $\MM_1$.
\end{theorem}

\begin{proof} If $[\scrM] = [\scrM']$ then \cref{Le:Char iso} directly gives $\LL\bdDelta_0(\scrM,\scrM') = 0$. Conversely, if $\LL\bdDelta_0(\scrM,\scrM')=0$, \cref{Th:ExOptL0q} easily implies the existence of a coupling $\pi$ of $\meas$ and $\meas'$ such that $\smash{\pi^{\otimes 2}}$-a.e.~$(x,x',y,y')\in (\mms\times\mms')^2$ obey $\smash{\vert \tau(x,y) - \tau(x',y')\vert = 0}$. Applying \cref{Le:Char iso} again yields $[\scrM]=[\scrM']$.

The symmetry of $\smash{\LL\bdDelta_0}$ is clear.

We now turn to the triangle inequality. Employing  \cref{Th:ExOptL0q}, let $\pi_{1,2}$ be an optimal coupling of $\meas$ and $\meas'$ in the definition of $\varepsilon := \LL\bdDelta_0(\scrM,\scrM')$. Analogously, let $\pi_{2,3}$ be an optimal coupling of $\meas'$ and $\meas''$ in the definition of $\varepsilon':=\LL\bdDelta_0(\scrM',\scrM'')$. Let $\alpha\in \Prob(\mms\times \mms'\times \mms'')$ be given by the gluing \cref{Le:Gluing}. Set $\pi_{1,3} := (\pr_1,\pr_3)_\push\alpha$, which is a coupling of $\meas$ and $\meas''$. We obtain
\begin{align*}
&\pi_{1,3}^{\otimes 2}\big[\big\{(x,x'',y,y'')\in (\mms\times\mms'')^2 : \big\vert\tau(x,y) - \tau''(x'',y'')\big\vert > \varepsilon + \varepsilon'\}\big]\\
&\qquad\qquad \leq \alpha^{\otimes 2}\big[\big\{(x,x',x'',y,y',y'')\in (\mms\times\mms'\times \mms'')^2 :\\
&\qquad\qquad\qquad\qquad \big\vert \tau(x,y) - \tau'(x',y')\big\vert + \big\vert\tau'(x',y') - \tau''(x'',y'')\big\vert > \varepsilon+\varepsilon'\big\}\big]\\
&\qquad\qquad \leq\pi_{1,2}^{\otimes 2}\big[\big\{(x,x',y,y')\in (\mms\times\mms')^2 : \big\vert\tau(x,y)-\tau'(x',y')\big\vert > \varepsilon\big\}\big]\\
&\qquad\qquad\qquad\qquad + \pi_{2,3}^{\otimes 2}\big[\big\{(x',x'',y',y'')\in (\mms'\times\mms'')^2 :\\
&\qquad\qquad\qquad\qquad\qquad\qquad\qquad \big\vert\tau'(x',y')-\tau''(x'',y'')\big\vert > \varepsilon'\big\}\big]\\
&\qquad\qquad \leq \varepsilon + \varepsilon',
\end{align*}
where we used \cref{Th:ExOptL0q} in the last bound. By definition of $\LL\bdDelta_0(\scrM,\scrM'')$,
\begin{align*}
\LL\bdDelta_0(\scrM,\scrM'') \leq \varepsilon +\varepsilon' = \LL\bdDelta_0(\scrM,\scrM') + \LL\bdDelta_0(\scrM',\scrM''),
\end{align*}
which is the desired inequality.

The first part of the proof together with the triangle inequality now make it clear \eqref{Eq:Llbox0} does not depend on the choice of representatives.
\end{proof}

\begin{definition}[Distortion convergence]\label{Def:Distortion convergence} We say $[\scrM_n]_{n\in\N}$ converges to $[\scrM_\infty]$ with  respect to $\LL\bdDelta_0$ provided
\begin{align*}
\lim_{n\to\infty} \LL\bdDelta_0(\scrM_n,\scrM_\infty)=0.
\end{align*}
\end{definition}

A simple property of the topology on $\MM_1$ induced by $\LL\bdDelta_0$ is the lower semicontinuity of diameters that we briefly discuss now. Given any bounded Lorentzian metric space $(\mms,\tau)$ and any subset $A\subset \mms$, define the \emph{diameter} of $A$ by
\begin{align*}
\diam A := \sup\!\big\{\tau(x,y) : x,y\in A\big\}.
\end{align*}

\begin{lemma}[Invariance of diameter under isomorphy]\label{Le:Invarisomdiam} Assume $\scrM$ and $\scrM'$ are isomorphic. Then $\diam\supp \meas = \diam\supp\meas'$.
\end{lemma}

\begin{proof} Given $\smash{L\in \R_+}$, by symmetry it suffices to show $\diam\supp\meas \leq L$ implies $\diam\supp\meas' \leq L$. Let $\pi$ be a coupling of $\meas$ and $\meas'$ provided by \cref{Le:Char iso}. Given any $\varepsilon >0$, a simple application of the triangle inequality and subadditivity yields
\begin{align*}
&\meas'^{\otimes 2}\big[\big\{(x',y')\in \mms'^2 : \tau'(x',y') > L +\varepsilon\big\}\big]\\
&\qquad\qquad \leq \pi^{\otimes 2}\big[\big\{(x,x',y,y')\in (\mms\times\mms')^2 : \big\vert\tau(x,y) - \tau'(x',y')\big\vert > \varepsilon\big\}\big]\\
&\qquad\qquad\qquad\qquad + \meas^{\otimes 2}\big[\big\{(x,y)\in \mms^2 : \tau(x,y) > L\big\}\big] =0.
\end{align*}
This shows $\tau'\leq L +\varepsilon$ $\smash{\meas'^{\otimes 2}}$-a.e.~and consequently $\tau'\leq L$ $\smash{\meas^{\otimes 2}}$-a.e.~thanks to the arbitrariness of $\varepsilon$. By continuity of $\tau'$, this implies $\diam\supp\meas' \leq L$.
\end{proof}

In particular, we can unambiguously define the diameter on $\MM_1$ by
\begin{align*}
\diam\,[\scrM] := \diam\supp\meas,
\end{align*}
where $\meas$ is the reference measure of any chosen representative $\scrM$.

A straightforward adaptation of the proof of  \cref{Le:Invarisomdiam} gives the following. To formulate the result, given any $L>0$ let $\MM_{1,L}$ denote the space of all $[\scrM]\in\MM_1$ such that $\diam\,[\scrM] \leq L$.

\begin{proposition}[Lower semicontinuity of diameter]\label{Pr:Diamlsc} Assume $[\scrM_n]_{n\in\N}$ converges to $[\scrM_\infty]$ with respect to $\LL\bdDelta_0$.  Then
\begin{align*}
\diam\,[\scrM]\leq \liminf_{n\to\infty} \diam\,[\scrM_n].
\end{align*}

In particular, for every $\smash{L\in\R_+}$ the space $\MM_{1,L}$ is closed with respect to $\smash{\LL\bdDelta_0}$.
\end{proposition}

\subsection{Lorentz--Gromov box distance} We conclude with a  variant of Gromov's famous \emph{box distance} for isomorphism classes of metric measure spaces \cite{Gro:99}*{§3$\smash{\frac{1}{2}}$.B}. Our presentation and several proofs  follow the exposition of Shioya \cite{Shioya2016}*{§4}.

Given  Borel measurable functions $\smash{u,v\colon [0,1]^2\to \R_+}$, we set
\begin{align}\label{Eq:squareuv}
\begin{split}
\square(u,v) := \inf\!\big\lbrace \varepsilon > 0 : \textnormal{there exists 
}Q \subset [0,1]\ \textnormal{Borel measurable}\qquad\qquad &\\
\textnormal{with }\scrL^1[Q] \leq  \varepsilon\textnormal{ and } \vert u - v\vert\leq \varepsilon \textnormal{ on }([0,1]\setminus Q)^2\big\rbrace.&
\end{split}
\end{align}
Then $\smash{\square}$ takes values in $[0,1]$. Furthermore, it is easy to infer $\square(t\,u,t\,v)$ depends nondecreasingly on $t>0$, while $\square(t\,u,t\,v)/t$ depends nonincreasingly on $t>0$.

Recall from \cref{Le:Parametrization} that every bounded Lorentzian metric measure space $\scrM$ has a parametrization  $\psi\colon[0,1]\to \mms$. We also record the definition \eqref{Eq:Pullback} of the induced pullback $\psi^*\tau\colon[0,1]^2\to \R_+$ of $\tau$.

\begin{definition}[Lorentz--Gromov box distance]\label{Def:LGboxdef}  We define the \emph{Lorentz--Gromov box distance} of $\scrM$ and $\scrM'$ by
\begin{align*}
\LL\BOX(\scrM,\scrM') := \inf\!\big\lbrace \square(\psi^*\tau,\psi'^*\tau') : \psi\in\Par\scrM,\,\psi'\in\Par\scrM'\big\rbrace.
\end{align*}
\end{definition}

\begin{remark}[Scaled Lorentz--Gromov box distances] Following the original definition of Gromov, we can also define a scaled version of the Lorentz--Gromov box distance. Let $\lambda >0$ be given. Replacing the condition $\smash{\Leb^1[Q] \leq \varepsilon}$ by $\smash{\Leb^1[Q]\leq \lambda\,\varepsilon}$ in \eqref{Eq:squareuv} defines the quantity $\square_\lambda(u,v)$ (so that $\square(u,v) = \square_1(u,v)$). In analogy to \cref{Def:LGboxdef}, we then define
\begin{align*}
\LL\BOX_\lambda(\scrM,\scrM') := \inf\!\big\lbrace \square_\lambda(\psi^*\tau,\psi'^*\tau') : \psi\in\Par\scrM,\,\psi'\in\Par\scrM'\big\rbrace.
\end{align*}
The simple inequalities
\begin{align*}
\square_\lambda \leq \square_{\lambda'} \leq \frac{\lambda}{\lambda'}\,\square_\lambda
\end{align*}
for every $\lambda'\in (0,\lambda)$ easily implies
\begin{align*}
\min\!\Big\{1, \frac{1}{\lambda}\Big\}\,\LL\BOX\leq \LL\BOX_\lambda \leq \max\!\Big\{1, \frac{1}{\lambda}\Big\}\,\LL\BOX.
\end{align*}
Hence, our notion from \cref{Def:LGboxdef} and $\LL\BOX_\lambda$ will induce the same topology on $\MM_1$. This allows us to reduce our considerations to $\LL\BOX$, but we note that its elementary properties also hold for $\LL\BOX_\lambda$.
\end{remark}

By a well-known rearrangement procedure  (cf.~e.g.~Brenier \cite{Bre:91} and the references therein), given any Borel measurable set $Q\subset [0,1]$, there exists a bijective Borel measurable map $u\colon[0,1]\to [0,1]$ with   $u_\push\Leb^1 = \Leb^1$ and $u([1-\Leb^1[Q],1])=Q$. Then the  simple identity
\begin{align}\label{Eq:Simple}
\square(\psi^*\tau,\psi'^*\tau') = \square((\psi\circ u)^*\tau,(\psi'\circ u)^*\tau')
\end{align}
yields the representation
\begin{align}\label{Eq:Alternative}
\begin{split}
\LL\BOX(\scrM,\scrM') &= \inf\!\big\lbrace\varepsilon > 0 : \textnormal{there exist }\psi\in\Par\scrM\textnormal{ and }\psi'\in\Par\scrM'\\
&\qquad\qquad \textnormal{with }\big\vert \psi^*\tau - \psi'^*\tau'\big\vert \leq \varepsilon \textnormal{ on }[0,1 -\varepsilon)^2\big\rbrace.
\end{split}
\end{align}
This should be compared to the following alternative representation of the distortion distance $\LL\bdDelta_0(\scrM,\scrM')$ in terms of parametrizations, which directly follows by combining \cref{Th:ExOptL0q} and \cref{Le:Parametrization}:
\begin{align*}
\LL\bdDelta_0(\scrM,\scrM') &= \inf\!\big\lbrace \varepsilon > 0 : \textnormal{there exist }\psi\in\Par\scrM \textnormal{ and }\psi'\in\Par\scrM'\\
&\qquad\qquad \textnormal{with } \Leb^2\big[\big\lbrace (s,t) \in [0,1]^2 :\\
&\qquad\qquad\qquad\qquad \big\vert \psi^*\tau(s,t) - \psi'^*\tau'(s,t)\big\vert \leq \varepsilon \big\rbrace\big] > 1-\varepsilon\big\rbrace.
\end{align*}
The difference is that the exceptional set in \eqref{Eq:Alternative} at scale $\varepsilon>0$ has a special form, namely it is the complement of a square of area $(1-\varepsilon)^2$.

Analogously to the preceding \cref{Sub:Lorentz Gromov Prokhorov}, we now verify that $\LL\BOX$ only depends on the isomorphism classes of its arguments and that it defines a metric on the space $\MM_1$. To this aim, we anticipatively define $\LL\BOX$ on $\MM_1$ by
\begin{align}\label{Eq:Boxwelldef}
\LL\BOX([\scrM],[\scrM']) := \LL\BOX(\scrM,\scrM').
\end{align}

\begin{theorem}[Metric property of $\LL\BOX$]\label{Th:MetricLBox} The quantity \eqref{Eq:Boxwelldef} constitutes a well-defined metric on $\MM_1$.
\end{theorem}

We prepare the proof of this theorem with some intermediate results.

The first shows the a priori not obvious necessity of  isomorphy for the vanishing of $\LL\BOX$. Alternatively, this could be seen from \cref{Th:Comparison} below and \cref{Le:Char iso}; we prefer to present an argument which does not rely on other metrics on $\MM_1$ but on our reconstruction \cref{Th:Gromov reconstruction} instead.

\begin{proposition}[Necessity of isomorphy]\label{Pr:Necessity} If $\LL\BOX(\scrM,\scrM')=0$, then $\scrM$ and $\scrM'$ are isomorphic.
\end{proposition}

\begin{proof}  We claim that the hypothesis implies $\smash{\bar{\meas}^k = \bar{\meas}'^k}$ for every $k\in\N$; our reconstruction \cref{Th:Gromov reconstruction} will then yield the desired isomorphy.

Given any sequence $(\varepsilon_n)_{n\in\N}$ of numbers in $(0,1)$  converging to zero and any $n\in\N$, there exist $\psi_n\in\Par\scrM$, $\psi_n'\in\Par\scrM'$, and  a Borel measurable set $Q_n \subset[0,1]$ such that $\Leb^1[Q_n]\leq \varepsilon_n$ and $\smash{\vert\psi_n^*\tau - \psi_n'^*\tau'\vert\leq \varepsilon_n}$ on $([0,1]\setminus Q_n)^2$. Let $\smash{\varphi\colon\R^{k\times k}\to\R}$ be bounded (say with $\vert\varphi\vert \leq C$ on $\R^{k\times k}$, where $C>0$) and uniformly continuous. Let $\omega\colon\R_+\to\R_+$ denote a modulus of continuity of $\varphi$ with respect to the $L^\infty$-norm $\Vert\cdot\Vert_\infty$ on $\smash{\R^{k\times k}}$. By \eqref{Eq:barm^k} and since $\psi_n$ and $\psi_n'$ parametrize $\meas$ and $\meas'$,
\begin{align*}
\int_{\R^{k\times k}} \varphi \d\bar{\meas}^k &= \int_{[0,1]^k} \varphi\circ \sfT^k \circ \psi_n^k \d\Leb^k,\\
\int_{\R^{k\times k}} \varphi \d\bar{\meas}'^k &= \int_{[0,1]^k} \varphi\circ \sfT'^k \circ \psi_n'^k \d\Leb^k,
\end{align*}
where $\smash{\psi_n^k\colon[0,1]^k\to \mms^k}$ and $\smash{\psi_n'^k\colon[0,1]^k\to\mms'^k}$ denote the $k$-fold product of $\psi_n$ and $\smash{\psi_n'}$, respectively. In other words, we have shifted the dependence of these terms  on $\scrM$ and $\scrM'$ from the reference measures to the function. Thus,
\begin{align*}
&\Big\vert\! \int_{\R^{k\times k}}\varphi\d\bar{\meas}^k -\int_{\R^{k\times k}}\varphi\d\bar{\meas}'^k\Big\vert\\
&\qquad\qquad \leq \Big\vert\!\int_{([0,1]\setminus Q_n)^k} \big[\varphi\circ \sfT^k \circ\psi_n^k - \varphi \circ \sfT'^k\circ\psi_n'^k\big]\d\Leb^k\Big\vert\\
&\qquad\qquad\qquad\qquad + \Big\vert\!\int_{[0,1]^k\setminus ([0,1]\setminus Q_n)^k} \big[\varphi\circ \sfT^k \circ\psi_n^k - \varphi \circ \sfT'^k\circ\psi_n'^k\big]\d\Leb^k\Big\vert\\
&\qquad\qquad \leq \omega(\varepsilon_n) + 2C\,\Leb^k\big[[0,1]^k\setminus ([0,1]\setminus Q_n)^k\big]\\
&\qquad\qquad \leq \omega(\varepsilon_n) + 2C\,k^2\,\varepsilon_n.
\end{align*}
Sending $n\to\infty$ establishes the claim.
\end{proof}

\begin{remark}[Uniform control] From the above proof, we also get  the following more general statement one can also easily make quantitative (which  will not be necessary in the discussion below). If $\scrF$ is a uniformly bounded family of functions on $\smash{\R^{k\times k}}$ with a common modulus of continuity, where $k\in\N$, then
\begin{align*}
\lim_{n\to\infty} \sup\!\Big\lbrace\Big\vert\!\int_{\R^{k\times k}}\varphi\d\bar{\meas}^k - \int_{\R^{k\times k}}\varphi\d\bar{\meas}'^k\Big\vert: \varphi\in\scrF\Big\rbrace = 0.
\end{align*}
Compare this with the classes from \cref{Re:Metrizability}.
\end{remark}

\begin{lemma}[Triangle inequality for $\square$]\label{Le:Trianglebox} Given any Borel measurable functions $u,v,w\colon [0,1]^2\to\R_+$,
\begin{align*}
\BOX(u,w) \leq \BOX(u,v) + \BOX(v,w).
\end{align*}
\end{lemma}

\begin{proof} We may and will assume   $\smash{\square(u,v) < 1}$ and $\smash{\square(v,w) < 1}$, otherwise the claimed inequality is trivial. Therefore, let $\varepsilon,\varepsilon'\in (0,1)$ satisfy $\square(u,v) <\varepsilon$ and $\square(v,w) < \varepsilon'$. Then there are Borel measurable sets $Q,Q'\subset[0,1]$ with $\Leb^1[Q] \leq \varepsilon$ and $\vert u-v\vert \leq\varepsilon$ on $([0,1]\setminus Q)^2$ as well as  $\Leb^1[Q'] \leq \varepsilon'$ and $\vert v-w\vert\leq \varepsilon'$ on $([0,1]\setminus Q')^2$. In particular, both estimates involving $u$, $v$, and $w$ hold on the smaller set $([0,1]\setminus Q'')^2$, where $Q'' := Q\cup Q'$. Noting that $\Leb^1[Q''] \leq \varepsilon + \varepsilon'$ and  $\vert u-w\vert  \leq \varepsilon + \varepsilon'$ on $([0,1]\setminus Q'')^2$ yields  $\square(u,w) \leq \varepsilon + \varepsilon'$, as desired.
\end{proof}

\begin{lemma}[Parametrizations with small distortion]\label{Le:Small dist} Let $\psi,\phi\in\Par\scrM$. Given any $\varepsilon >0$, there exist bijective Borel measurable functions $u,v\colon [0,1]\to[0,1]$ such that $\smash{u_\push\Leb^1=\Leb^1}$, $\smash{v_\push\Leb^1 = \Leb^1}$, and
\begin{align*}
\Box((\psi\circ u)^*\tau, (\phi\circ v)^*\tau) \leq \varepsilon.
\end{align*}
\end{lemma}

\begin{proof} Let us consider the distinction metric $\sfn$. Following the argument of Shioya \cite{Shioya2016}*{Lem.~4.9} for this metric measure space, we build bijective Borel measurable maps $u,v\colon[0,1]\to [0,1]$ with $\smash{u_\push\Leb^1 = \Leb^1}$, $\smash{v_\push\Leb^1=\Leb^1}$, and $\smash{\sfn(\psi\circ u, \phi\circ v) \leq \varepsilon/2}$ on  $[0,1]$. By Lipschitz continuity of $\tau$ with respect to $\smash{\sfn^{\oplus 2}}$, every $s,t\in[0,1]$ obey
\begin{align*}
&\big\vert \tau(\psi\circ u(s),\psi\circ u(t)) - \tau(\phi\circ v(s),\phi\circ v(t))\big\vert\\
&\qquad\qquad \leq \sfn(\psi\circ u(s),\phi\circ v(s)) + \sfn(\psi\circ u(t),\phi\circ v(t))\\
&\qquad\qquad \leq \varepsilon.
\end{align*}
This finishes the proof.
\end{proof}

\begin{proof}[Proof of \cref{Th:MetricLBox}] If $\LL\BOX(\scrM,\scrM')=0$, \cref{Pr:Necessity} directly  implies $[\scrM]=[\scrM']$.  Conversely, assume $[\scrM]=[\scrM']$. Then combining \cref{Le:Char iso,Le:Parametrization}, there are $\psi\in\Par\scrM$ and $\psi'\in\Par\scrM'$ such that $\psi^*\tau = \psi'^*\tau'$ $\Leb^2$-a.e. Modifying $\psi$ and $\psi'$ appropriately according to the comment before \eqref{Eq:Alternative} yields $\LL\BOX(\scrM,\scrM') =0$.

Symmetry of $\LL\BOX$ follows from the definition.

Let us now come to the triangle inequality. Let $\psi\in\Par\scrM$, $\psi',\phi'\in\Par\scrM'$, and $\psi''\in\Par\scrM''$ be given. Given any $\varepsilon > 0$, let $u,v\colon[0,1]\to[0,1]$ be associated to $\psi'$ and $\phi'$ according to \cref{Le:Small dist}. Using \cref{Le:Trianglebox} and \eqref{Eq:Simple}, we obtain
\begin{align*}
&\square(\psi^*\tau,\psi'^*\tau') + \square(\phi'^*\tau',\psi''^*\tau')\\
&\qquad\qquad = \square((\psi\circ u)^*\tau,(\psi'\circ u)^*\tau') + \square((\phi'\circ v)^*\tau',(\psi''\circ v)^*\tau'')\\
&\qquad\qquad \geq \square((\psi\circ u)^*\tau,(\psi''\circ v)^*\tau'') -\varepsilon\\
&\qquad\qquad \geq \LL\BOX(\scrM,\scrM'')-\varepsilon.
\end{align*}
The arbitrariness of $\psi$, $\psi'$, $\phi'$, $\psi''$, and $\varepsilon$ gives the claim.

Combining the first part of the proof with the triangle inequality, it is finally standard to see the definition \eqref{Eq:Boxwelldef} is well-posed, i.e.~independent of the precise choice of representatives.
\end{proof}

\begin{definition}[Lorentz--Gromov box convergence]\label{DeF:BoxConv} We say that  $[\scrM_n]_{n\in\N}$ \emph{box-converges} to $[\scrM_\infty]$ provided
\begin{align*}
\lim_{n\to\infty} \LL\BOX(\scrM_n,\scrM_\infty) = 0.
\end{align*}
\end{definition}

\begin{remark}[Nonnormalized bounded Lorentzian metric measure spaces]\label{Re:Finite meas sp!} In the preceding part, we have restricted ourselves to normalized bounded Lorentzian metric measure spaces. With evident adaptations, all our results hold when the references measures are merely finite. Our Gromov reconstruction \cref{Th:Gromov reconstruction finite measure} and \cref{Def:Intrinsic} carry over unchanged. For instance, in the proof of \cref{Th:Gromov reconstruction finite measure} one would have to normalize the measures, but observe that either hypothesis of \cref{Th:Gromov reconstruction finite measure} implies the masses in question coincide.

In \cref{Def:L0dist}, one has to account for different masses separately.  Let $\scrM$ and $\scrM'$ be finite bounded Lorentzian metric measure spaces. Define $\smash{\scrN := (\mms,\tau,\mathfrak{n})}$ and $\smash{\smash{\scrN'} := (\mms',\tau',\mathfrak{n}')}$, where $\mathfrak{n} := \meas[\mms]^{-1}\,\meas$ and $\mathfrak{n}' := \meas'[\mms']^{-1}\,\meas'$. Then set
\begin{align*}
\LL\bdDelta_0(\scrM,\scrM') := \LL\bdDelta_0(\scrN,\scrN') + \Big\vert\!\log \frac{\meas[\mms]}{\meas'[\mms']}\Big\vert.
\end{align*}
It is straightforward to see this extension defines a metric on the set of isomorphism classes of finite bounded Lorentzian metric measure spaces.

A similar discussion applies to \cref{Def:LGboxdef} above and \cref{Sub:Strong} below.
\end{remark}

\subsection{Mutual relations} Now we prove fundamental relations between the above three notions of convergence. This establishes a hierarchy between them according to the next two theorems.

\begin{theorem}[From distortion to intrinsic convergence]\label{Pr:IntrDist} Assume the sequence $[\scrM_n]_{n\in\N}$ converges $[\scrM_n]_{n\in\N}$ converges to $[\scrM_\infty]$  with respect to $\LL\bdDelta_0$. Then this convergence also holds intrinsically. 
\end{theorem}

\begin{proof} Let $k\in\N$ be fixed throughout the sequel.  Let $\varepsilon > 0$. For every sufficiently large $n\in\N$ there exists a coupling $\pi_n$ of $\meas_n$ and $\meas_\infty$ such that
\begin{align}\label{Eq:leqeps}
\begin{split}
&\pi_n^{\otimes 2}\big[\big\lbrace(x_{n,1},x_{\infty,1},x_{n,2},x_{\infty,2})\in (\mms_n\times\mms_\infty)^2 :\\
&\qquad\qquad  \big\vert\tau_n(x_{n,1},x_{n,2}) - \tau_\infty(x_{\infty,1},x_{\infty,2})\big\vert > \varepsilon \big\rbrace\big] \leq \varepsilon.
\end{split}
\end{align}
Subadditivity and \eqref{Eq:leqeps} easily imply
\begin{align}\label{Eq:pikneps}
\pi_n^{\otimes k}\big[(\mms_n\times\mms_\infty)^k \setminus E\big] \leq k^2\,\varepsilon
\end{align}
for the closed set
\begin{align*}
E &:= \big\lbrace (x_{n,1},x_{\infty,1},\dots,x_{n,k},x_{\infty,k})\in (\mms_n\times\mms_\infty)^k :\\
&\qquad\qquad \big\vert\tau_n(x_{n,i},x_{n,j}) - \tau_\infty(x_{\infty,i}, x_{\infty,j})\big\vert \leq \varepsilon \textnormal{ for every }i,j\in\{1,\dots,k\}\big\rbrace.
\end{align*}

By \cref{Re:Testfcts}, it suffices to show  convergence of $\smash{(\bar{\meas}_n^k)_{n\in\N}}$ to $\smash{\bar{\meas}_\infty^k}$ in duality with bounded and $\met_\infty$-Lipschitz continuous functions on $\smash{\R^{k\times k}}$. Let $\varphi$ be such a  function. Let $C>0$ be an  upper bound on $\vert\varphi\vert$ on $\smash{\R^{k\times k}}$. By definition of $E$ and \eqref{Eq:pikneps},
\begin{align*}
&\Big\vert\!\int_{\R^{k\times k}} \varphi \d\bar{\meas}_n^k - \int_{\R^{k\times k}} \varphi\d\bar{\meas}_\infty^k\Big\vert\\
&\qquad\qquad = \Big\vert\!\int_{\mms_n^k} \varphi\circ \sfT_n^k \d\meas_n^{\otimes k} - \int_{\mms_\infty^k}\varphi\circ \sfT_\infty^k \d\meas_\infty^{\otimes k}\Big\vert \\
&\qquad\qquad = \Big\vert\!\int_{(\mms_n\times\mms_\infty)^k} \d\pi_n^{\otimes k}(x_{n,1},x_{\infty,1},\dots,x_{n,k},x_{\infty,k})\\ 
&\qquad\qquad\qquad\qquad \big[\varphi\circ \sfT_n^k(x_{n,1},\dots,x_{n,k}) -  \varphi\circ \sfT_\infty^k(x_{\infty,1},\dots,x_{\infty,k})\big]\Big\vert\\
&\qquad\qquad \leq \int_E \d\pi_n^{\otimes k}(x_{n,1},x_{\infty,1},\dots,x_{n,k},x_{\infty,k})\\ 
&\qquad\qquad\qquad\qquad\qquad\qquad \big\vert \varphi\circ \sfT_n^k(x_{n,1},\dots,x_{n,k}) -  \varphi\circ \sfT_\infty^k(x_{\infty,1},\dots,x_{\infty,k})\big\vert \\
&\qquad\qquad\qquad\qquad + \int_{(\mms_n\times\mms_\infty)^k\setminus E} \d\pi_n^{\otimes k}(x_{n,1},x_{\infty,1},\dots,x_{n,k},x_{\infty,k})\\ 
&\qquad\qquad\qquad\qquad\qquad\qquad \big\vert \varphi\circ \sfT_n^k(x_{n,1},\dots,x_{n,k}) -  \varphi\circ \sfT_\infty^k(x_{\infty,1},\dots,x_{\infty,k})\big\vert\\
&\qquad\qquad \leq \big[\Lip\,\varphi + 2C\,k^2\big]\,\varepsilon.
\end{align*}
This concludes the proof.
\end{proof}

\begin{theorem}[From box to distortion distance]\label{Th:Comparison} We have 
\begin{align*}
\LL\bdDelta_0 &\leq \LL\BOX.
\end{align*}

In particular, if $[\scrM_n]_{n\in\N}$ box-converges to $[\scrM_\infty]$, this convergence holds in the intrinsic sense and with respect to $\LL\bdDelta_0$ as well.
\end{theorem}

\begin{proof} The second claim follows from the first and \cref{Pr:IntrDist}.

For the first statement, without loss of generality we may and will assume that $\LL\BOX(\scrM,\scrM') < 1$. Let $\varepsilon\in (0,1)$ such that $\LL\BOX(\scrM,\scrM') < \varepsilon$.  Then there exist parametrizations $\psi\in\Par\scrM$ and $\psi'\in\Par\scrM'$ and a Borel measurable set $Q\subset[0,1]$ such that $\smash{\Leb^1[Q]\leq \varepsilon <  \sqrt{\varepsilon}}$ and $\smash{([0,1]\setminus Q)^2\subset  E}$, where 
\begin{align*}
E := \big\lbrace(s,t)\in [0,1]^2 : \big\vert\psi^*\tau(s,t)  - \psi'^*\tau'(s,t)\big\vert \leq  \varepsilon\big\rbrace.
\end{align*}
Thanks to \cref{Le:Parametrization}, the assignment $\pi := (\psi,\psi')_\push\Leb^1$ constitutes a coupling of $\meas$ and $\meas'$. Given any $\delta > 0$,  we obtain
\begin{align*}
&\pi^{\otimes 2}\big[\big\{(x,x',y,y')\in(\mms\times\mms')^2 : \big\vert\tau(x,y) - \tau'(x',y')\big\vert \geq (1+\delta)\varepsilon\big\}\\
&\qquad\qquad \leq \Leb^2\big[[0,1]^2\setminus E\big]\\
&\qquad\qquad \leq \Leb^2\big[(Q\times [0,1]) \cup ([0,1]\times Q)\big]\\
&\qquad\qquad \leq \varepsilon.
\end{align*}
The claim follows from the arbitrariness of  $\varepsilon$.
\end{proof}

\appendix

\section{Further distances}\label{Sub:Strong}

In \cref{Sub:mLGH} we focused on the most natural notions of ``weak'' convergence of elements of $\MM_1$. In this appendix, we introduce further distances of different strengths that are certain variants of those already introduced. One of these distances will be an ``$L^\infty$-distance'', which we relate to the Lorentz--Gromov--Hausdorff convergence of Müller \cite{Mue:22} and Minguzzi--Suhr \cite{MS:22b} in  \cref{Sub:RelGromHaus}.

\subsection{$L^p$-distortion distance} The first mirrors the $L^p$-distortion distance for isomorphism classes of normalized  metric measure spaces by Mémoli \cite{Mem:11}*{Def.~5.7} and Sturm \cite{Stu:12}*{Def.~1.6}, where $p\in[1,\infty]$. As the approach suggests, this defines ``$L^p$-convergence'' of a sequence $[\scrM_n]_{n\in\N}$ in the space $\MM_1$ to $[\scrM_\infty]$.

\begin{definition}[$L^p$-distortion distance]\label{Def:Lpdistdist} Given any $p\in[1,\infty)$, let us define the \emph{$L^p$-distortion distance} of $\scrM$ and $\scrM'$ by
\begin{align*}
\LL\bdDelta_p(\scrM,\scrM') &:= \inf\Big[\!\int_{(\mms\times\mms')^2} \big\vert \tau(x,y) -\tau'(x',y')\big\vert^p\d\pi^{\otimes 2}(x,x',y,y')\Big]^{1/p},
\end{align*}
where the infimum ranges over all couplings $\pi$ of $\meas$ and $\meas'$.

Moreover, the \emph{$L^\infty$-distortion distance} of $\scrM$ and $\scrM'$ is given by
\begin{align*}
\LL\bdDelta_\infty(\scrM,\scrM') := \inf \pi^{\otimes 2}\textnormal{-}\!\esssup\!\big\lbrace \big\vert\tau(x,y) - \tau'(x',y')\big\vert: (x,x',y,y') \in (\mms\times\mms')^2 \big\rbrace,
\end{align*}
where the infimum is again taken over all couplings $\pi$ of $\meas$ and $\meas'$.
\end{definition}

All these distances are finite, given the respective time separation functions of the bounded Lorentzian metric measure spaces in question are.

As for \cref{Th:ExOptL0q}, these quantities admit optimal couplings.

\begin{theorem}[Existence of optimal couplings]\label{Th:ExOptLp} For every $p\in[1,\infty]$, there is a coupling $\pi$ of $\meas$ and $\meas'$ that attains the respective infimum in the above definition of $\LL\bdDelta_p(\scrM,\scrM')$.
\end{theorem}

\begin{proof} Irrespective of whether $p$ is finite or infinite, let $(\pi_n)_{n\in\N}$ be a given sequence of  couplings of $\meas$ and $\meas'$, where $\pi_n$  is $2^{-n}$-almost optimal for the definition of the quantity $\smash{\LL\bdDelta_p(\scrM,\scrM')}$ for every $n\in\N$. As in the proof of \cref{Th:ExOptL0q}, by simple compactness arguments a nonrelabeled subsequence of  $(\pi_n)_{n\in\N}$ converges to a coupling $\pi$ of $\meas$ and $\meas'$. We claim $\pi$ is optimal for $\smash{\LL\bdDelta_p(\scrM,\scrM')}$.

If $p$ is finite,  observe the integrand $\smash{\vert \tau\circ(\pr_1,\pr_3)- \tau'\circ(\pr_2,\pr_4)\vert^p}$ is continuous and nonnegative. Hence, the integral minimized in the definition of $\smash{\LL\bdDelta_p(\scrM,\scrM')}$ depends narrow lower semicontinuously on (the two-fold product of) the coupling, cf.~e.g.~Villani \cite{Vil:09}*{Lem.~4.3}. With the definition of $\smash{\LL\bdDelta_p(\scrM,\scrM')}$ and the almost optimality of $(\pi_n)_{n\in\N}$, this implies
\begin{align*}
\LL\bdDelta_p(\scrM,\scrM') &\leq \Big[\!\int_{(\mms\times\mms')^2} \big\vert\tau(x,y) - \tau'(x',y')\big\vert^p\d\pi^{\otimes 2}(x,x',y,y')\Big]^{1/p}\\
&\leq\liminf_{n\to\infty} \Big[\!\int_{(\mms\times\mms')^2} \big\vert\tau(x,y) - \tau'(x',y')\big\vert^p\d\pi_n^{\otimes 2}(x,x',y,y')\Big]^{1/p}\\
&= \LL\bdDelta_p(\scrM,\scrM').
\end{align*}
This forces equality throughout, implying optimality of $\pi$.

We now address the existence of an optimal coupling for  $\LL\bdDelta_\infty(\scrM,\scrM')$. Given any $\varepsilon > 0$ and any $n\in\N$ with $\smash{2^{-n} \leq \varepsilon}$, by construction we have
\begin{align*}
\pi_n^{\otimes 2}\textnormal{-}\!\esssup\!\big\lbrace  \big\vert\tau(x,y)-\tau'(x',y')\big\vert : (x,x',y,y')\in (\mms\times\mms')^2 \big\rbrace \leq \LL\bdDelta_\infty(\scrM,\scrM') + \varepsilon.
\end{align*}
By lower semicontinuity of open sets under narrow convergence, this entails
\begin{align*}
&\pi^{\otimes 2}\big[\big\{(x,x',y,y') \in (\mms\times\mms')^2 : \big\vert\tau(x,y)-\tau'(x',y')\big\vert > \LL\bdDelta_\infty(\scrM,\scrM')+\varepsilon\big\}\big]\\
&\qquad\qquad\leq \liminf_{n\to\infty} \pi_n^{\otimes 2}\big[\big\{(x,x',y,y') \in (\mms\times\mms')^2 :\\
&\qquad\qquad\qquad\qquad\qquad\qquad \big\vert\tau(x,y)-\tau'(x',y')\big\vert > \LL\bdDelta_\infty(\scrM,\scrM')+\varepsilon\big\}\big]\\
&\qquad\qquad =0.
\end{align*}
Combining this with the definition of $\LL\bdDelta_\infty(\scrM,\scrM')$, 
\begin{align*}
\LL\bdDelta_\infty(\scrM,\scrM') &\leq \pi^{\otimes 2}\textnormal{-}\!\esssup\!\big\{\big\vert\tau(x,y)-\tau'(x',y')\big\vert : (x,x',y,y')\in(\mms\times\mms')^2 \big\}\\
&\leq \LL\bdDelta_\infty(\scrM,\scrM') + \varepsilon.
\end{align*}
The arbitrariness of $\varepsilon$ concludes the argument.
\end{proof}

With this theorem at hand, the proof of the following consequence is analogous to the argument for \cref{Th:Metric prop} involving gluing techniques, hence omitted.

Given any $p\in [1,\infty]$, we anticipatively define $\LL\bdDelta_p$ on $\MM_1$ by
\begin{align}\label{Eq:Llboxp}
\LL\bdDelta_p([\scrM],[\scrM']) := \LL\bdDelta_p(\scrM,\scrM').
\end{align}

\begin{theorem}[Metric property of $\LL\bdDelta_p$] For every  $p\in[1,\infty]$, the quantity \eqref{Eq:Llboxp} constitutes a well-defined metric on $\MM_1$.
\end{theorem}

Lastly, we compare the topologies on $\MM_1$ induced by the distortion distance $\LL\bdDelta_0$ from \cref{Def:L0dist} and by $\LL\bdDelta_p$, where $p\in[1,\infty]$. Along the way, we relate the distances $\smash{\LL\bdDelta_p}$ and $\smash{\LL\bdDelta_{p'}}$ for different $p,p'\in[1,\infty]$.

\begin{proposition}[Markov's inequality]\label{Pr:Markov} For every $p\in[1,\infty]$,
\begin{align*}
\LL\bdDelta_0 \leq \sqrt{\LL\bdDelta_p}.
\end{align*}

In particular, the convergence of $[\scrM_n]_{n\in\N}$ to $[\scrM_\infty]$ with respect to $\smash{\LL\bdDelta_p}$ implies convergence of $[\scrM_n]_{n\in\N}$ to $[\scrM_\infty]$ with respect to $\smash{\LL\bdDelta_0}$ for every $p\in[1,\infty]$.
\end{proposition}

\begin{proof} The first claim is clear if $\smash{\LL\bdDelta_p(\scrM,\scrM') \geq 1}$. Otherwise, we first focus on the case $p\in[1,\infty)$ and let $\varepsilon\in (0,1)$ with $\smash{\LL\bdDelta_p(\scrM,\scrM') <\varepsilon}$. By definition, this gives the existence of a coupling $\pi$ of $\meas$ and $\meas'$ with the property
\begin{align*}
\int_{(\mms\times\mms')^2} \big\vert\tau(x,y) - \tau'(x',y')\big\vert^p\d\pi^{\otimes 2}(x,x',y,y') \leq \varepsilon^p.
\end{align*}
Markov's inequality then yields
\begin{align*}
&\pi^{\otimes 2}\big[\big\{(x,x',y,y')\in (\mms\times\mms')^2 : \big\vert\tau(x,y) - \tau'(x',y')\big\vert >\sqrt{\varepsilon} \big\}\big]\\
&\qquad\qquad \leq \frac{1}{\sqrt{\varepsilon}^p}\int_{(\mms\times\mms')^2} \big\vert\tau(x,y) - \tau'(x',y')\big\vert^p\d\pi^{\otimes 2}(x,x',y,y')\\
&\qquad\qquad \leq \sqrt{\varepsilon}^p\\
&\qquad\qquad \leq \sqrt{\varepsilon}.
\end{align*}
The arbitrariness of $\varepsilon$ yields the claim.

The estimate $\smash{\LL\bdDelta_0 \leq \sqrt{\LL\bdDelta_\infty}}$ follows from \cref{Pr:Compdist} below.

The last statement follows now directly.
\end{proof}

Recall the space $\MM_{1,L}$ from before \cref{Pr:Diamlsc}, where $L>0$.

\begin{proposition}[Comparison of $L^p$-distortion distances  across transport exponents]\label{Pr:Compdist} Given any $p,p'\in [1,\infty]$ with $p<p'$, we have
\begin{align*}
\LL\bdDelta_p  \leq \LL\bdDelta_{p'}.
\end{align*}

Moreover, $\LL\bdDelta_\infty$ is the monotone limit of $\LL\bdDelta_p$ as $p\to\infty$.

Lastly, given any $p,p'\in[1,\infty)$ with $p<p'$, on $\smash{\MM_{1,L}^2}$ we have the estimate
\begin{align*}
\LL\bdDelta_{p'} \leq L^{1-p/p'}\,\LL\bdDelta_p^{p/p'}
\end{align*}
\end{proposition}

\begin{proof} The first and third estimate follow by applying Jensen's inequality or the trivial estimate of the $L^p$- by the $L^\infty$-norm to the optimal coupling for $\smash{\LL\bdDelta_{p'}(\scrM,\scrM')}$ and $\smash{\LL\bdDelta_p(\scrM,\scrM')}$, respectively, as provided by \cref{Th:ExOptLp}.

The second estimate is argued in a standard way as follows. Let $(p_n)_{n\in\N}$ be any sequence in $[1,\infty)$ increasing to infinity. Let $(\pi_n)_{n\in\N}$ be a sequence of couplings of $\meas$ and $\meas'$ such that $\pi_n$ is optimal for $\smash{\LL\bdDelta_{p_n}(\scrM,\scrM')}$. By the same  argument as for  \cref{Th:ExOptL0q}, we may and will assume that a nonrelabeled subsequence of $(\pi_n)_{n\in\N}$ converges to a coupling $\pi$ of $\meas$ and $\meas'$. By narrow lower semicontinuity of the involved integrals, Jensen's inequality, the increasing  convergence $p_n\to \infty$ as $n\to\infty$, and the first statement, every $p\in [1,\infty)$ satisfies
\begin{align*}
&\Big[\!\int_{(\mms\times\mms')^2} \big\vert\tau(x,y) -\tau'(x',y')\big\vert^p\d\pi^{\otimes 2}(x,x',y,y')\Big]^{1/p}\\
&\qquad\qquad \leq \liminf_{n\to\infty} \Big[\!\int_{(\mms\times\mms')^2} \big\vert\tau(x,y) -\tau'(x',y')\big\vert^p\d\pi_n^{\otimes 2}(x,x',y,y')\Big]^{1/p}\\
&\qquad\qquad \leq \liminf_{n\to\infty} \LL\bdDelta_{p_n}(\scrM,\scrM')\\
&\qquad\qquad \leq \limsup_{n\to\infty} \LL\bdDelta_{p_n}(\scrM,\scrM')\\
&\qquad\qquad \leq \LL\bdDelta_\infty(\scrM,\scrM').
\end{align*}
Sending $p\to \infty$ on the left-hand side yields
\begin{align*}
\LL\bdDelta_\infty(\scrM,\scrM') &\leq \pi^{\otimes 2}\textnormal{-}\!\esssup\!\big\{\big\vert\tau(x,y) - \tau'(x',y')\big\vert : (x,x',y,y')\in (\mms\times\mms')^2\big\}\\
&\leq \liminf_{n\to\infty} \LL\bdDelta_{p_n}(\scrM,\scrM')\\
&\leq \limsup_{n\to\infty} \LL\bdDelta_{p_n}(\scrM,\scrM')\\
&\leq \LL\bdDelta_\infty(\scrM,\scrM').
\end{align*}
This forces equality to hold throughout.
\end{proof}

\begin{corollary}[Nestedness of topologies across transport exponents]\label{Cor:Nest} Given any $p,p'\in[1,\infty]$ with $p< p'$, if $[\scrM_n]_{n\in\N}$ converges to $[\scrM_\infty]$ with respect to $\smash{\LL\bdDelta_{p'}}$, it does so with respect to $\smash{\LL\bdDelta_p}$.

The converse holds if there exists $L>0$ with $[\scrM_n]\in \MM_{1,L}$ for every $n\in\N$.
\end{corollary}

Note that in the last statement of this corollary, the limit $[\scrM_\infty]$ belongs to $\MM_{1,L}$ as well. This follows from \cref{Pr:Markov,Pr:Diamlsc}.

Lastly, we discuss sufficient conditions when $\LL\bdDelta_0$-convergence implies $\LL\bdDelta_p$-convergence for  $p\in[1,\infty)$, following Sturm \cite{Stu:12}*{§2.1} in metric measure geometry. Since \cref{Th:ExOptLp} is at our disposal, the proof of the following proposition consists of hardly more than the following well-known relation from probability theory it mimics, cf.~e.g.~Billingsley \cite{Bil:99}*{§3}: the $L^p$-convergence of a sequence of functions on a given probability space is equivalent to their convergence in measure plus convergence of their $p$-th moments or their uniform $p$-integrability. The proof of the following result is analogous to Sturm \cite{Stu:12}*{Prop.~2.1}, hence omitted.

\begin{proposition}[From $L^0$- to $L^p$-distortion convergence]\label{Pr:FromL0toLp} Given any  $p\in[1,\infty)$, the following statements are equivalent.
\begin{enumerate}[label=\textnormal{(\roman*)}]
\item The sequence $[\scrM_n]_{n\in\N}$ converges to $[\scrM_\infty]$ with respect to $\LL\bdDelta_p$.
\item The sequence $[\scrM_n]_{n\in\N}$ converges to $[\scrM_\infty]$ with respect to $\LL\bdDelta_0$ and
\begin{align*}
\lim_{n\to\infty} \int_{\mms_n^2} \tau_n^p\d\meas_n^{\otimes 2} = \lim_{n\to\infty} \int_{\mms_\infty^2}\tau_\infty^p \d\meas_\infty^{\otimes 2}.
\end{align*}
\item The sequence $[\scrM_n]_{n\in\N}$ converges to $[\scrM_\infty]$ with respect to $\LL\bdDelta_0$ and
\begin{align*}
\lim_{L\to\infty} \sup_{n\in\N}\int_{\{\tau_n > L\}} \tau_n^p\d\meas_n^{\otimes 2} = 0.
\end{align*}
\end{enumerate}
\end{proposition}

\subsection{Distortions with different exponents}\label{Sub:L0qLpq} Raising the given time separation functions in \cref{Def:L0dist,Def:Lpdistdist} to a power different from one produces another set of distances on $\MM_1$ we briefly discuss here. For metric measure spaces, these variants were introduced and studied by Sturm \cite{Stu:12}*{§9}. We omit most proofs since they are straightforward adaptations of the preceding arguments.

\begin{definition}[$L^{0,q}$-distortion distance] Given any $q\in [1,\infty)$, we define the \emph{$L^{0,q}$-distortion distance} of $\scrM$ and $\scrM'$ by
\begin{align*}
\LL\bdDelta_{0,q}(\scrM,\scrM') := \inf \inf\!\big\lbrace\varepsilon > 0 : \pi^{\otimes 2}\big[\big\lbrace (x,x',y,y') \in (\mms\times\mms')^2 :\qquad\qquad & \\
\big\vert\tau(x,y)^q - \tau'(x',y')^q\big\vert \geq \varepsilon\big\rbrace\big] \leq \varepsilon\big\},&
\end{align*}
where the outer infimum is taken over all couplings $\pi$ of $\meas$ and $\meas'$.
\end{definition}

\begin{definition}[$L^{p,q}$-distortion distance] Given any $q\in[1,\infty)$ and and $p\in[1,\infty)$, we define the \emph{$L^{p,q}$-distortion distance} of $\scrM$ and $\scrM'$ by
\begin{align*}
\LL\bdDelta_{p,q}(\scrM,\scrM') := \inf \Big[\!\int_{(\mms\times\mms')^2}  \big\vert\tau(x,y)^q  - \tau'(x',y')^q\big\vert^p\Big]^{1/p} \d\pi^{\otimes 2}(x,x',y,y'),
\end{align*}
where the infimum is taken over all couplings $\pi$ of $\meas$ and $\meas'$.

Moreover, the \emph{$L^{\infty,q}$-distortion distance} of $\scrM$ and $\scrM'$ is given by
\begin{align*}
\LL\bdDelta_{\infty,q}(\scrM,\scrM') := \inf \pi^{\otimes 2}\textnormal{-}\!\esssup\!\big\lbrace \big\vert\tau(x,y)^q - \tau'(x',y')^q\big\vert :\qquad\qquad &\\
 (x,x',y,y')\in (\mms\times\mms')^2\big\rbrace,&
\end{align*}
where the infimum is taken over all couplings $\pi$ of $\meas$ and $\meas'$.
\end{definition}

As both $\tau$ and $\tau'$ are bounded, the distances introduced above are all finite; in fact, the $L^{0,q}$-distortion distance even takes values in $[0,1]$.

\begin{theorem}[Existence of optimal couplings] For every $p\in \{0\}\cup[1,\infty]$ and every $q\in[1,\infty)$, there exists a coupling $\pi$ of $\meas$ and $\meas'$ which is optimal in the respective definition of $\smash{\LL\bdDelta_{p,q}(\scrM,\scrM')}$. 
\end{theorem}

In turn, given any $p\in\{0\}\cup[1,\infty]$ and any $q\in[1,\infty)$, we anticipatively define $\LL\bdDelta_{p,q}$ on $\MM_1$ by
\begin{align}\label{Eq:Lpbdelta}
\LL\bdDelta_{p,q}([\scrM'],[\scrM]) := \LL\bdDelta_{p,q}(\scrM,\scrM').
\end{align}

\begin{theorem}[Metric property of $\smash{\LL\bdDelta_{p,q}}$] For every $p\in\{0\}\cup[1,\infty]$ and every $q\in[1,\infty)$, the quantity \eqref{Eq:Lpbdelta} constitutes a well-defined metric on $\MM_1$.
\end{theorem}

When the distortion exponent $q\in[1,\infty)$ is fixed, analogous comparison results hold for the distances $\smash{\LL\bdDelta_{p,q}}$ across $p\in\{0\}\cup[1,\infty]$ as those from  \cref{Pr:Markov,Pr:Compdist,Cor:Nest,Pr:FromL0toLp}. Thus it remains to discuss the dependency on $q$. To this aim, we recall the space $\MM_{1,L}$ from before \cref{Pr:Diamlsc}, where $L>0$.

\begin{proposition}[Comparison of distortion distances] The distance $\LL\bdDelta_{0,q}^{1/q}$ depends in a nondecreasing way on $q\in [1,\infty)$. 

Moreover, given any $L>0$ and any $q,q'\in [1,\infty)$, on $\smash{\MM_{1,L}^2}$ we have 
\begin{align*}
\frac{q}{q'}\,L^{1-q/q'}\,\LL\bdDelta_{0,q'} \leq \LL\bdDelta_{0,q}.
\end{align*}
\end{proposition}

\begin{proof} To prove the first statement, let $q,q'\in [1,\infty)$ such that $q<q'$. Let $\pi$ be an optimal coupling for the definition of $\LL\bdDelta_{0,q'}(\scrM,\scrM')$ according to \cref{Th:ExOptL0q}. Defining $\smash{\varepsilon_0 := \LL\bdDelta_{0,q'}(\scrM,\scrM')}$ --- which is not larger than one --- and finally using the inequality $\smash{\vert a^q-b^q\vert^{1/q}\leq \vert a^{q'}-b^{q'}\vert^{1/q'}}$ for every $a,b\in\R_+$, we obtain
\begin{align*}
\varepsilon_0^{q/q'} &\geq \varepsilon_0\\
&\geq \pi^{\otimes 2}\big[\big\lbrace (x,x',y,y')\in (\mms\times\mms')^2 : \big\vert \tau(x,y)^{q'} - \tau'(x',y')^{q'}\big\vert > \varepsilon_0\big\rbrace\big]\\
&= \pi^{\otimes 2}\big[\big\lbrace (x,x',y,y')\in (\mms\times\mms')^2 : \big\vert \tau(x,y)^{q'} - \tau'(x',y')^{q'}\big\vert^{q/q'} > \varepsilon_0^{q/q'}\big\rbrace\big]\\
&\geq \pi^{\otimes 2}\big[\big\lbrace (x,x',y,y')\in (\mms\times\mms')^2 : \big\vert \tau(x,y)^q - \tau'(x',y')^q\big\vert > \varepsilon_0^{q/q'}\big\rbrace\big]
\end{align*}
The claim follows from the definition of $\smash{\LL\bdDelta_{0,q}(\scrM,\scrM')}$.

The second claim follows from a similar argument using the simple inequality $\smash{\vert a^r - b^r\vert^{1/r} \leq r\,L^{r-1}\,\vert a-b\vert}$ for every $a,b\in[0,L]$ and every $r\in[1,\infty)$.
\end{proof}

We summarize parts of the above observations in the following result. Note that by the same argument as for \cref{Pr:Diamlsc}, the diameter is lower semicontinuous with respect to any of the distances mentioned in this corollary.

\begin{corollary}[Uniformly bounded diameters]\label{Cor:Nesteddd}  Let $L>0$ and assume $[\scrM_n]_{n\in\N}$ is a sequence in $\MM_{1,L}$. Then given any $p\in[1,\infty)$ and any $q\in[1,\infty)$, the following statements are equivalent.
\begin{enumerate}[label=\textnormal{(\roman*)}]
\item The sequence $[\scrM_n]_{n\in\N}$ converges to $[\scrM_\infty]$ with respect to $\LL\bdDelta_0$.
\item The sequence $[\scrM_n]_{n\in\N}$ converges to $[\scrM_\infty]$ with respect to $\LL\bdDelta_p$.
\item The sequence $[\scrM_n]_{n\in\N}$ converges to $[\scrM_\infty]$ with respect to $\LL\bdDelta_{0,q}$.
\item The sequence $[\scrM_n]_{n\in\N}$ converges to $[\scrM_\infty]$ with respect to $\LL\bdDelta_{p,q}$.
\end{enumerate}
\end{corollary}

\subsection{Relation to Lorentz--Gromov--Hausdorff distance and convergence}\label{Sub:RelGromHaus} The Lorentz--Gromov--Hausdorff semidistance  introduced by Müller \cite{Mue:22}*{§2} and Minguzzi--Suhr \cite{MS:22b}*{Def.~4.6} applies to bounded Lorentzian metric spaces \emph{without} a priori given reference measures.  The purpose of this part is to relate their Lorentz--Gromov--Hausdorff distance to the $L^\infty$-distortion distance from \cref{Def:Lpdistdist}. In view of the previous  \cref{Sub:L0qLpq}, an analogous relation --- that we do not comment on further to simplify the presentation --- holds between $\smash{\LL\bdDelta_{\infty,q}}$ and an evident adaptation of \eqref{Eq:LGHd} below using the $q$-distortion for a given $q\in[1,\infty)$.

Recall that a correspondence of bounded Lorentzian metric spaces $(\mms,\tau)$ and $(\mms',\tau')$ \cite{MS:22b}*{Def.~4.1} is a set $R\subset\mms\times \mms'$ with $\pr_1(R) = \mms$ and $\pr_2(R) =\mms'$. Then, following Minguzzi--Suhr \cite{MS:22b}*{Def.~4.3}, let us define the distortion of such a correspondence by 
\begin{align*}
\dis\, R := \sup\!\big\lbrace\big\vert\tau(x,y) - \tau'(x',y')\big\vert : (x,x'),(y,y')\in R\big\rbrace.
\end{align*}

The \emph{Lorentz--Gromov--Hausdorff semidistance} $\met_{\mathrm{LGH}(\mms,\mms')}$ of $(\mms,\tau)$ and $(\mms',\tau')$ as defined by Minguzzi--Suhr \cite{MS:22b}*{Def.~4.6}  is then
\begin{align}\label{Eq:LGHd}
\met_{\mathrm{LGH}}(\mms,\mms') := \inf\!\big\{\dis\, R: R\subset\mms\times\mms' \textnormal{ closed correspondence}\big\}.
\end{align}
Requiring closedness of the correspondences will be convenient  below, yet is inessential to the above definitions: taking the closure of a correspondence does not change its distortion \cite{MS:22b}*{Prop.~4.5}. Recall that $\smash{\met_{\mathrm{LGH}}}$ defines a semimetric on the set of bounded Lorentzian metric spaces (and a genuine metric after quotienting out isometries), cf.~Minguzzi--Suhr \cite{MS:22b}*{Thm.~4.13}.

\begin{proposition}[Control by $L^\infty$-distortion distance]\label{Pr:LGHcomp} Let $\scrM$ and $\scrM'$ be two fully supported normalized bounded Lorentzian metric measure spaces.  Then
\begin{align*}
\met_{\mathrm{LGH}}(\mms,\mms') \leq \LL\bdDelta_\infty(\scrM,\scrM').
\end{align*}
\end{proposition}

Since the distance $\smash{\met_{\mathrm{LGH}}}$ takes into account the entire spaces (and not merely  supports of reference measures),  we cannot expect a convergence comparison  whose hypotheses are unchanged under isomorphisms. An interesting question we do not address here is to show the equivalence of unmeasured and measured Lorentz--Gromov--Hausdorff convergence under appropriate hypotheses. Likely, this will involve a suitable notion of uniform doubling throughout the sequence in question, a first approach to which is developed by McCann--Sämann \cite{McCS:22}. See Villani \cite{Vil:09} and Gigli--Mondino--Savaré \cite{GMS:15} for metric measure geometry.

Details for the subsequent proof can be found in Mémoli \cite{Mem:11}*{§2.1}, which are of purely probabilistic nature and thus carry over to our setting. 

\begin{proof}[Proof of \cref{Pr:LGHcomp}] Every given coupling $\pi$ of $\meas$ and $\meas'$ induces a closed correspondence  by $R:= \supp\pi$, cf.~Mémoli \cite{Mem:11}*{Lem.~2.2}. Since $\supp\meas = \mms$ and $\supp\meas' = \mms'$ and since $\tau$ and $\tau'$ are continuous,
\begin{align*}
\met_{\mathrm{LGH}}(\mms,\mms') &\leq \dis\,R\\
&= \pi^{\otimes 2}\textnormal{-}\!\esssup\!\big\{\big\vert\tau(x,y) - \tau'(x',y')\big\vert : (x,x',y,y')\in (\mms\times\mms')^2\big\}.
\end{align*}
The arbitrariness of $\pi$ yields the claim.
\end{proof}

\bibliographystyle{amsrefs}


\end{document}